\documentclass[10pt,a4paper]{amsart}
\usepackage{amsmath,mathtools,amsthm,amssymb,color,latexsym, tikz, longtable, cite, float, subcaption, hyperref}
\usepackage[sort&compress,numbers]{natbib}
\usepackage{ytableau}
\usepackage{xcolor}  
\usepackage{pdfpages}
\usepackage{rotating}
\usepackage{pdflscape}
\usepackage{diagbox}
\usepackage{booktabs}
\usepackage{hyperref}
\usepackage{adjustbox}
\usepackage[normalem]{ulem}

\usepackage[lmargin=28mm,rmargin=28mm,tmargin=35mm,bmargin=35mm]{geometry}

\newcommand{\spacing}[1]{
\renewcommand{\baselinestretch}{#1}
\setlength{\footnotesep}{\baselinestretch\footnotesep}}
\spacing{1.15}

\theoremstyle{plain}
\newtheorem{theorem}{Theorem}[section]  
\newtheorem{proposition}[theorem]{Proposition}  
\newtheorem{corollary}[theorem]{Corollary}  
  
\newtheorem{definition}[theorem]{Definition}  
\newtheorem{remark}[theorem]{Remark} 
\newtheorem{example}[theorem]{Example}  
\newtheorem{question}[theorem]{Question}

\newcommand{\balpha}[0]{\boldsymbol{\alpha} }

\newcommand{\pr}[1]{\left(#1\right)}
\newcommand{\PS}{\mathsf{P\Lambda}}
\newcommand{\x}[1]{\mathbf{x}_{#1*}}
\newcommand{\area}[1]{\operatorname{area}{(#1)}}
\newcommand{\sgn}[1]{\operatorname{sgn}(#1)}
\newcommand{\sect}[1]{\hyperref[section #1]{Section #1}}

\begin{document}

\title[Decomposition of Polysymmetric Functions and Stack Partitions
]{Decomposition of Polysymmetric Functions and Stack Partitions
}
\author[]{David Martinez}
\thanks{}

\keywords{polysymmetric functions, symmetric functions, stack partitions, generating functions, Jacobi–Trudi identities, involution}

\begin{abstract}
Polysymmetric functions, introduced by Asvin G and Andrew O’Desky as a generalization of symmetric functions, have natural connections to algebraic geometry and provide a foundation for further developments.
In this paper, we study polysymmetric functions using stack partitions and develop combinatorial descriptions of several polysymmetric bases.
We introduce two new signed polysymmetric bases and give explicit transition formulas among the monomial, homogeneous, elementary, power, and signed polysymmetric bases.
These results extend many familiar identities from symmetric function theory to the polysymmetric setting.
\end{abstract}

\maketitle
\section{Introduction}

Polysymmetric functions were introduced by Asvin G and Andrew O'Desky
in \cite{g2024polysymmetricfunctionsmotivicmeasures} as a generalization of the
algebra $\Lambda$ of symmetric functions.
They defined several polysymmetric bases for the graded algebra $\PS$ and
introduced the notion of a \emph{type} $\tau$, which plays a role analogous to
that of a partition in classical symmetric function theory.
These types coincide with objects that have appeared previously in the
literature under the name \emph{factorization patterns} of $n$.
Such patterns were studied by Hultquist, Mullen, and Niederreiter
\cite{MR944348} in the context of association schemes and derived PBIB designs, and later by
Agarwal and Mullen \cite{Agarwal1988PartitionsW}, who gave combinatorial
interpretations and connections to plane partitions.
Since then, related structures have appeared in a variety of settings,
including lattice paths \cite{Sahibpreet}, mock theta functions
\cite{Brietzke2011}, split perfect partitions \cite{goyal2019analogue},
beaded partitions \cite{MR2956334}, scalar polynomial functions
\cite{MR1176446}, and non-squashing partitions \cite{FOLSOM20161482}.

The work of Asvin G and Andrew O'Desky was motivated in part by applications
to algebraic geometry, with an emphasis on connections to the Grothendieck
ring of varieties.
Subsequently, Khanna and Loehr \cite{khanna2024transitionmatricespierityperules}
studied transition matrices and multiplication rules among several
polysymmetric bases, focusing primarily on bases arising from pure tensor
constructions.

This paper continues the study of polysymmetric functions from a combinatorial
perspective.
We adopt the term \emph{stack partitions} for types and use this language
throughout.
In \sect{2}, we introduce polysymmetric functions and stack partitions.
In \sect{3}, we give a new combinatorial interpretation of several
polysymmetric bases in terms of monomial polysymmetric functions, obtained
by counting certain integer matrices.
In \sect{4}, we introduce two new signed polysymmetric bases and establish
their properties.
In \sect{5}, we define a signed involution on $\PS$ and derive Jacobi--Trudi type identities, determining the action of this involution on the remaining known polysymmetric bases and obtaining explicit transition formulas relating the complete homogeneous, elementary, and power bases of the polysymmetric algebra, together with their signed variants.
In \sect{6}, we pose structural questions for polysymmetric functions, consider related questions for stack partitions and their generalizations, and record initial values of an associated enumerative sequence.

\section{Introduction to Polysymmetric Functions and Stack Partitions}
\label{section 2}
In this section, we introduce polysymmetric functions and stack partitions, assuming the reader has some familiarity with the theory of symmetric functions and their standard bases $e_\lambda, h_\lambda, p_\lambda, m_\lambda,s_\lambda$ (see \cite{Stanley_Fomin_1999,Macdonald}).
\subsection{Polysymmetric Functions}
Let $\Lambda=\Lambda_\mathbb{Q}$ denote the $\mathbb{Q}$-algebra of symmetric functions. For each positive integer $m$, let $\Lambda_{(m)}$ be a copy of $\Lambda$ in which each homogeneous symmetric function of degree $n$ in $\Lambda$ is regarded as having degree $mn$. The $\mathbb{Q}$-algebra of polysymmetric functions is defined as the tensor product,
\begin{align}
    \PS=\Lambda_{(1)}
\otimes \Lambda_{(2)}
\otimes \Lambda_{(3)} \otimes \cdots.
\end{align}

Then $\PS$ can be viewed as having the set of indeterminates $\x{*}= \{x_{m, j}\}_{m,j=1}^\infty$, where we say $x_{m,j}$ has degree $m$ and view $\Lambda_{(m)}$ as the ring of symmetric functions in a variable set $\x{m}=\{x_{m, j}\}_{j=1}^\infty$. 
Let $\mathbb{Q}[[\x{*}]]$ denote the ring of all formal infinite sums of monomials in $\x{*}$ with rational coefficients and bounded degree. 
A formal series $f=f(\x{*})$ is in $\PS$ if for each fixed $m$, $f$ remains unchanged by any permutation of the variables in $\x{m}$. 
An isomorphism between the tensor and series versions of $\PS$ is defined by sending the pure tensor $f_1\otimes f_2\otimes f_3\otimes\cdots$ to the formal series $f_1(\x{1}) f_2(\x{2})f_3(\x{3})\cdots$. Therefore, $\PS$ has the structure of a graded algebra

\begin{align}
    \PS = \bigoplus_{n\geq0}\PS^n,
\end{align}
where $\PS^n$ is the vector space of homogeneous polysymmetric functions of degree $n$.
\begin{example}
Let
\begin{align}
    f
= \sum_{i,k} x_{2,i}^2 x_{3,k}^2
  + \sum_{i<j, \, k} x_{2,i}x_{2,j}x_{3,k}^2
  + \sum_{i,j, k,\ell} x_{1,i}x_{2,j}x_{3,k}x_{4,\ell}.
\end{align}
Each monomial is symmetric in the variable sets $x_{1,*},x_{2,*},x_{3,*},x_{4,*}$, thus $f\in \PS$.
Since every term of $f$ has degree $10$, we also have $f\in \PS^{10}$.
\end{example}

\subsection{Stack Partitions}
 A \emph{stack} is an ordered pair of positive integers $(d, m)$, which we write as $d^m$. We say that $d^m$ has degree $d$ and  multiplicity $m$. We order stacks by having $d_i^{m_i} \geq d_j^{m_j}$ if $d_i > d_j$, or if $d_i = d_j$ and $m_i \geq m_j$. A \emph{stack partition} of $n$ is a weakly decreasing sequence of stacks $\tau = ( d_1^{m_1}, d_2^{m_2}, \ldots, d_s^{m_s})$
such that $d_1 m_1 + d_2 m_2 + \ldots + d_s m_s = n$.
We write $\tau \Vdash n$ or $|\tau|=n$ to indicate that $\tau$ is a stack partition of $n$. 
When every stack in $\tau$ has multiplicity $1$ (i.e., $\tau$ is an ordinary partition), we write $\tau \vdash n$.

%\sout{A stack partition can also be constructed by starting with a partition $\lambda = (\lambda_1, \lambda_2, \ldots, \lambda_s)$ of $n$ and choosing a \emph{divisor tuple} $k = (k_1, k_2, \ldots, k_s)$ such that each $k_i$ divides $\lambda_i$. Replacing each $\lambda_i$ with the stack $k_i^{\lambda_i / k_i}$ gives the stack partition $\tau = ( k_1^{\lambda_1 / k_1}, k_2^{\lambda_2 / k_2}, \ldots, k_s^{\lambda_s / k_s} ).$ The same procedure applies even when we start with a stack partition instead of an ordinary partition.}

For any positive integer $r$, we define the \emph{stack partition of $n^r$} as the stack partition of $n$ with the multiplicities of each stack scaled by $r$.

\begin{example}
The stack partitions of $4$ and $4^3$ are given by
\begin{align*}
&4: \quad 4, \quad 3 1, \quad 2 2, \quad 2 1 1, \quad 1 1 1 1, \quad 2^2, \quad 2 1^2, \quad 1^2 1 1, \quad 1^2 1^2, \quad 1^3 1, \quad 1^4; \\
&4^3: \quad 4^3, \quad 3^3 1^3, \quad 2^3 2^3, \quad 2^3 1^3 1^3, \quad 1^3 1^3 1^3 1^3, \quad 2^6, \quad 2^3 1^6, \quad 1^6 1^3 1^3, \quad 1^6 1^6, \quad 1^9 1^3, \quad 1^{12}.
\end{align*}
\end{example}

The \emph{length} of a stack partition $\tau$ is the number of stacks it contains, denoted by $\ell(\tau)$. For a fixed $d$, let $\tau \mid_d$ denote the partition consisting of the multiplicities of the stacks with degree $d$ in $\tau$. Let $\operatorname{area}(\tau \mid_d)$ be the sum of the multiplicities of the stacks with degree $d$ in $\tau$, and define $\operatorname{area}(\tau)$ as the total sum of all multiplicities within $\tau$, so that $\operatorname{area}(\tau) = \sum_{i=1}^{\ell(\tau)} m_i$.
We define the \emph{sign} of the stack partition $\tau$ as
\begin{align}
\label{eqn:sgn}
\operatorname{sgn}(\tau) = \prod_{i=1}^{k} (-1)^{\operatorname{area}(\tau \mid_{i})} = (-1)^{\operatorname{area}(\tau)}.
\end{align}

\begin{example}
Let $\tau = (3^3, 3^2, 3^2, 2^7, 2^2, 2^1, 1^5, 1^1)$, so that $\ell(\tau)=8$, 
$\tau|_3=(3,2,2)$, $\tau|_2=(7,2,1)$, and $\tau|_1=(5,1)$. 
Furthermore, $\operatorname{area}(\tau)=23$ and $\operatorname{sgn}(\tau)=(-1)^{23}=-1$.
\end{example}

Our interpretation of stack partitions aligns with \cite{g2024polysymmetricfunctionsmotivicmeasures, khanna2024transitionmatricespierityperules}, where they are referred to as \emph{types}. We use the term \emph{stack partitions} to emphasize their construction from ordinary partitions by visualizing each number as a box and allowing the action of ``stacking'' boxes with the same number. This construction connects naturally to the concepts of stack, degree, and multiplicity introduced earlier, where each box represents a stack, its label gives the degree, and the number of boxes in each column gives its multiplicity. After making a stacking, we order the resulting stacks as described earlier in the section.

\begin{example}

We start with the partition $(4,2,2,2,1,1,1)$, we obtain the stack partition

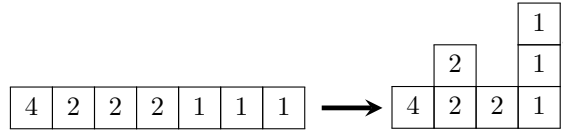
\begin{center}
\begin{tikzpicture}[>=stealth, thick, scale=0.8]

\node (left) at (-1.3,0) {$
\begin{ytableau}
4 & 2 & 2 & 2 & 1 & 1 & 1
\end{ytableau}$};

\node (right) at (4,0.7) {$
\begin{ytableau}
\none &\none&\none&1 \\
\none &2&\none&1 \\
4& 2 & 2&1 
\end{ytableau},$};

\draw[->, line width=1.8pt] (1.4,0) -- (2.4,0);
\end{tikzpicture}
\end{center}
which corresponds to $(4,2^2,2,1^3)$. Here, the stack partition is obtained by placing one of the twos on top of another two and stacking all three of the ones together, followed by the ordering.
\end{example}

\subsection{Polysymmetric Bases}
In \cite{g2024polysymmetricfunctionsmotivicmeasures}, four bases of $\PS$ are introduced, denoted by $\{H_\tau\}$, $\{E^+_\tau\}$, $\{E_\tau\}$, and $\{P_\tau\}$, which are non-tensor bases. We provide interpretations in terms of stack partitions using the monomial polysymmetric function $M_\tau$, also introduced in \cite{g2024polysymmetricfunctionsmotivicmeasures}, which is a pure tensor product.
\begin{definition}
    Given a stack partition $\tau = (d_1^{m_1}, d_2^{m_2}, \ldots, d_s^{m_s}) \Vdash n$, define the \textbf{monomial polysymmetric function $M_\tau(\mathbf{x})$ of type $\tau$} by
    \begin{align}
         M_\tau(\mathbf{x}) = \sum_{\alpha} \mathbf{x}^\alpha,
    \end{align}
    where each monomial $\mathbf{x}^\alpha$ is of the form
    \begin{align}
         \mathbf{x}^\alpha = x_{d_1, \alpha_1}^{m_1} x_{d_2, \alpha_2}^{m_2} \cdots x_{d_s, \alpha_s}^{m_s}.
    \end{align}
    The sum is taken over all sequences $\alpha = (\alpha_1, \ldots, \alpha_s)$ of positive integers such that $\alpha_i \ne \alpha_j$ whenever $d_i = d_j$ and $i \ne j$. 
\end{definition}

\begin{remark}
    We have that $M_\tau(\mathbf{x})$ equals $m_{\tau\mid _1}(\x{1})m_{\tau\mid_2}(\x{2})\cdots m_{\tau\mid_s}(\x{s})$, where $m_{\lambda}$ is the monomial symmetric function indexed by the partition $\lambda$.
\end{remark}
\begin{example}
Let $\tau = (2^2,1^3,1^2,1^2)$, then $M_\tau$ is the sum
    \begin{align*}
        M_{2^21^31^21^2}(\mathbf{x})= m_{(3,2,2)}(\mathbf{x}_{1*})m_{(2)}(\mathbf{x}_{2*}) = \sum_{\substack{i<j,\; k,l\\ k\notin\{i,j\}}}
 x_{1,i}^2 x_{1,j}^2 x_{1,k}^3x_{2,l}^2.
    \end{align*}
\end{example}
We define four bases of $\PS$, denoted by $\{H_\tau\}$, $\{E^+_\tau\}$, $\{E_\tau\}$, and $\{P_\tau\}$. These serve as polysymmetric analogues of the symmetric bases $\{h_\lambda\}$, $\{e_\lambda\}$, and $\{p_\lambda\}$. 

The $d$\textbf{th complete homogeneous polysymmetric function} is 
\begin{equation} 
H_d = 
\sum_{\alpha \Vdash d } M_{\alpha}.
\end{equation} 

The $d$\textbf{th unsigned elementary polysymmetric function} is
\begin{equation} %{*
E^+_{d} = \sum_{\alpha\vdash d}  M_\alpha .
\end{equation} 

The $d$\textbf{th elementary polysymmetric function} is 
\begin{equation} %{*
E_d =
\sum_{\alpha\vdash d} 
\operatorname{sgn}(\alpha)M_\alpha.
\end{equation} 

The $d$\textbf{th power polysymmetric function} is
\begin{equation} 
P_{d} = \sum_{k\mid d} k M_{k^{d/k}},
\end{equation} 
where $k\mid d$ means $k$ divides $d$.
\begin{example}
    Let \( d = 4 \). Then the expansions in the monomial polysymmetric basis are given by
    \begin{align}
        H_4 &= 
        M_{4} 
        + M_{31} 
        + M_{22} 
        + M_{2^2} 
        + M_{211} 
        + M_{21^2} 
        + M_{1111}+ M_{1^2 11} 
        + M_{1^2 1^2}
        + M_{1^3 1}  
        + M_{1^4}, \\
        E^+_4 &= 
        M_{4} 
        + M_{31} 
        + M_{22} 
        + M_{211} 
        + M_{1111}, \\
        E_4 &= 
        -M_{4} 
        + M_{31} 
        + M_{22} 
        - M_{211} 
        + M_{1111}, \\
        P_4 &= 
        4M_{4} + 2M_{2^2} + M_{1^4}.
    \end{align}
\end{example}
We define $H_{d^m}(\x{*})=H_d(\x{*}^m)$, meaning every variable $x_{i,j}$ appearing in every monomial of $H_d$ gets replaced by $x_{i,j}^m$. We define similarly $E^+_{d^m}(\x{*}) = E^+_{d}(\x{*}^m), E_{d^m}(\x{*}) = E_d(\x{*}^m)$, and $P_{d^m}(\x{*}) = P_d(\x{*}^m)$, where all belong in $\PS^{dm}$. Lastly, for any stack partition $\tau = (d_1^{m_1}, d_2^{m_2},\ldots, d_s^{m_s})$, define
\begin{align}
    H_{\tau} = \prod_{i=1}^{t} H_{d_i^{m_i}}, \quad
E_{\tau}^+ = \prod_{i=1}^{t} E^+_{d_i^{m_i}}, \quad
E_{\tau} = \prod_{i=1}^{t} E_{d_i^{m_i}}, \quad
P_{\tau} = \prod_{i=1}^{t} P_{d_i^{m_i}}.
\end{align}
It was shown in \cite{g2024polysymmetricfunctionsmotivicmeasures} that each of the sets $\{H_{\tau}: \tau\Vdash n\}, \{E^+_{\tau}:
\tau\Vdash n\}, \{E_{\tau}: \tau\Vdash n\}$, and $\{P_{\tau}: \tau\Vdash n\}$ is a linear basis of $\PS^n$. 

The generating functions for the polysymmetric bases $H$, $E^+$, $E$, and $P$ 
appear in \cite{g2024polysymmetricfunctionsmotivicmeasures}. 
We provide a short proof for completeness.

\begin{proposition}[\cite{g2024polysymmetricfunctionsmotivicmeasures}]
The generating functions for $H$, $E^+$, $E$, and $P$ are given by
\begin{align}
    \label{gen:H}
    \mathcal{H}(t) &= \sum_{d=0}^\infty H_{d^{m}}  t^{md} = \prod_{i,j} (1 - x_{i, j}^m t^{m i})^{-1}, \\
    \label{gen:EP}
    \mathcal{E}^+(t) &= \sum_{d=0}^\infty E^+_{d^{m}}  t^{md} = \prod_{i,j} (1 + x_{i, j}^m t^{m i}), \\
    \label{gen:E}
    \mathcal{E}(t) &= \sum_{d=0}^\infty E_{d^{m}}  t^{md} = \prod_{i,j} (1 - x_{i, j}^m t^{m i}), \\
    \label{gen:P}
    \mathcal{P}(t) &= \sum_{d=1}^\infty  P_{d^{m}} t^{md - 1} = \sum_{d=1}^\infty  \pr{ \sum_{i \mid d} i \sum_{j=1}^\infty x_{i, j}^{md/i}} t^{md - 1}.
\end{align}
\end{proposition}

\begin{proof}
    Since $F_{d^m}(\mathbf{x}_*) = F_d(\mathbf{x}_*^m)$ 
for $F \in \{H, E^+, E, P\}$, the case $m>1$ 
is obtained from $m=1$ by the substitution 
$x_{i,j} \mapsto x_{i,j}^m$ and $t \mapsto t^m$. 
It therefore suffices to consider $m=1$. To verify \eqref{gen:H}, observe that for each pair $(i,j)$,
\begin{align}
    (1 - x_{i,j} t^i)^{-1} = \sum_{k=0}^\infty x_{i,j}^k t^{ik},
\end{align}
as a formal power series in the indeterminate $t$.
Multiplying over all pairs $(i,j)$ enumerates every possible choice of 
variables $x_{i,j}$, allowing repetitions, and therefore yields 
$\sum_{d=0}^\infty H_d t^d$ as the generating function. For \eqref{gen:EP}, expanding $\prod_{i,j}(1 + x_{i,j} t^i)$ 
produces monomials that are square-free in the variables $x_{i,j}$, 
yielding $\sum_{d=0}^\infty E^+_d t^d$. Formula \eqref{gen:E} follows from \eqref{gen:EP} 
after the substitution $x^m_{i,j} \mapsto -x^m_{i,j}$, 
and \eqref{gen:P} follows directly from 
the definition of the polysymmetric power sums.
\end{proof}
\section{A New Combinatorial Interpretation of the Expansion in the $M$ Basis}
\label{section 3}

We now present a new combinatorial interpretation of the expansions of the four bases in terms of the monomial polysymmetric functions. This result was previously established by the authors of \cite{khanna2024transitionmatricespierityperules} using brick tabloids. Here, we provide an alternative interpretation by counting certain matrices.
\subsection{The Product of Polysymmetric Monomials $M$}
In this subsection, we give a matrix counting rule for products of monomial polysymmetric functions.

\begin{definition}
Let $\balpha=(\alpha_1, \alpha_2,\ldots,\alpha_r)$ be a sequence of stack partitions and fix 
$\tau=(d_1^{m_1},\ldots,d_s^{m_s})$. 
Define $s^\tau_{\balpha}(u)$ to be the number of $\mathbb{N}$-matrices with $r$ rows and $\ell(\tau\mid_u)$ columns such that the $i$th row is a permutation of $\alpha'_i\mid_u$ and whose sequence of column sums equals $\tau\mid_u$.
Here $\alpha'_i\mid_u$ denotes the sequence obtained from $\alpha_i\mid_u$ with trailing zeros (if necessary) so that its length equals $\ell(\tau\mid_u)$.
If any $\ell(\alpha_i\mid_u)>\ell(\tau\mid_u)$, then $s_{\balpha}^\tau(u)=0$.
\end{definition}

We combine these counts over all values of $u$ by defining
\begin{align}
    \mathcal{S}^{\tau}_{\balpha} = \prod_{u=1}^\infty s^\tau_{\balpha}(u)
\end{align}

\begin{example}
\label{ex:1}
Let $\balpha=(\alpha_1,\alpha_2,\alpha_3)$, where 
$\alpha_1=(2,2,1^2,1)$, $\alpha_2=(2,1)$, and $\alpha_3=(3^3,1)$, 
and fix $\tau=(3^3,2,2,2,1^2,1^2,1)$. 
We compute $s^\tau_{\balpha}(1)$ by considering $\mathbb{N}$-matrices with three rows and 
column-sum sequence $\tau\mid_1=(2,2,1)$. 
Each row is a permutation of $\alpha'_i\mid_1$, hence the first, second, and third rows permute $(2,1,0)$, $(1,0,0)$, and $(1,0,0)$, respectively.
The valid matrices are
\begin{align}
\begin{bmatrix} 2 & 1 & 0 \\ 0 & 1 & 0 \\ 0 & 0 & 1 \end{bmatrix},\quad
\begin{bmatrix} 2 & 1 & 0 \\ 0 & 0 & 1 \\ 0 & 1 & 0 \end{bmatrix},\quad
\begin{bmatrix} 2 & 0 & 1 \\ 0 & 1 & 0 \\ 0 & 1 & 0 \end{bmatrix},\quad
\begin{bmatrix} 1 & 2 & 0 \\ 1 & 0 & 0 \\ 0 & 0 & 1 \end{bmatrix},\quad
\begin{bmatrix} 1 & 2 & 0 \\ 0 & 0 & 1 \\ 1 & 0 & 0 \end{bmatrix},\quad
\begin{bmatrix} 0 & 2 & 1 \\ 1 & 0 & 0 \\ 1 & 0 & 0 \end{bmatrix}.
\end{align}
so $s^\tau_{\balpha}(1)=6$.
Next, $s^\tau_{\balpha}(2)$ counts $\mathbb{N}$-matrices with column-sum sequence $\tau\mid_2=(1,1,1)$.
Here the first, second, and third rows permute $(1,1,0)$, $(1,0,0)$, and $(0,0,0)$, respectively.
The valid matrices are
\begin{align}
\begin{bmatrix} 0 & 1 & 1 \\ 1 & 0 & 0 \\ 0 & 0 & 0 \end{bmatrix}, \quad
\begin{bmatrix} 1 & 0 & 1 \\ 0 & 1 & 0 \\ 0 & 0 & 0 \end{bmatrix}, \quad
\begin{bmatrix} 1 & 1 & 0 \\ 0 & 0 & 1 \\ 0 & 0 & 0 \end{bmatrix}.
\end{align}
Thus $s^\tau_{\balpha}(2)=3$. Finally, for $s^\tau_{\balpha}(3)$, the only valid matrix is
\begin{align}
\begin{bmatrix} 0 \\ 0 \\ 3 \end{bmatrix},
\end{align}
so $s^\tau_{\balpha}(3)=1$. 
Therefore $\mathcal{S}^{\tau}_{\balpha} = 18.$
\end{example}

This leads to the following theorem, which expresses the coefficients appearing in products of monomial polysymmetric functions in terms of the matrix counts defined above.

\begin{theorem}
Let $\balpha=(\alpha_1,\ldots,\alpha_r)$ be a sequence of stack partitions and  
$n=\sum_{i=1}^r |\alpha_i|$. Then
\begin{align}
    \label{thm:MtimesM}
       \prod^r_{i=1}{M_{\alpha_i}} = \sum_{\tau\Vdash n} \mathcal{S}_{\balpha}^{\tau} M_{\tau}.
    \end{align}
\end{theorem}
\begin{proof}
Fix a positive integer $u$ and a stack partition $\tau\Vdash n$. Write $\balpha=(\alpha_1,\ldots,\alpha_r)$.  
For each $i$, choose a term from $M_{u^{\alpha_i\mid_u}}$ of the form
\begin{align}
    x_{u,1}^{a_{i1}}x_{u,2}^{a_{i2}}\cdots x_{u,\ell}^{a_{i\ell}},
\end{align}
where $\ell=\ell(\tau\mid_u)$ and $(a_{i1},\ldots,a_{i\ell})$ is a permutation of $\alpha'_i\mid_u$. Then a term in $\prod_{i=1}^s M_{u^{\alpha_i\mid_u}}$ is obtained by choosing terms $x_{u,1}^{a_{i1}}x_{u,2}^{a_{i2}}\cdots x_{u,\ell}^{a_{i\ell}}$ from each $M_{u^{\alpha_i\mid_u}}$ such that
\begin{align}
    \prod_{i=1}^r \big(x_{u,1}^{a_{i1}}x_{u,2}^{a_{i2}}\cdots x_{u,\ell}^{a_{i\ell}}\big)
 = \prod_{j=1}^{\ell} x_{u,j}^{\sum_{i=1}^r a_{ij}}.
\end{align}
 A contribution to the coefficient of $M_{u^{\tau\mid_u}}$ occurs precisely when
\begin{align}
    \big(\sum_i a_{i1},\ldots,\sum_i a_{i\ell}\big)=\tau\mid_u.
\end{align}
 This is the same as choosing an $\mathbb N$-matrix $A=(a_{ij})$ whose $i$th row is a permutation of $\alpha'_i\mid_u$ and whose column-sum sequence equals $\tau\mid_u$. The number of such choices is $s^{\tau}_{\balpha}(u)$ by definition.
Variables with distinct degrees $u\neq u'$ are disjoint, so choices at different degrees are independent and hence we can multiply. Therefore
\begin{align}
    \prod_{i=1}^r M_{\alpha_i}
= \sum_{\tau}\Big(\prod_{u} s^{\tau}_{\balpha}(u)\Big)\,M_{\tau} = \sum_{\tau}S_{\balpha}^{\tau}M_{\tau},
\end{align}
and the coefficient of $M_\tau$ is $S_{\balpha}^\tau$, as claimed.
\end{proof}

\begin{example}
    Let $\alpha_1 = (2,2,1^2,1)$, $\alpha_2 = (2,1)$, and $\alpha_3 = (3^3,1)$, as given in Example~\ref{ex:1}. Then
    \begin{equation}
        \begin{aligned}
        M_{221^21} \cdot M_{21} \cdot M_{3^3 1} =\ 
        &18 M_{3^3 222 1^2111} +
         6\, M_{3^3 2^2 2 1^2111} +
        12 M_{3^3 222 1^3 11} + 
         4 M_{3^3 2^2 2 1^31 1}\\
        &+ 18M_{3^3 2 221^21^2 1} +
         6\, M_{3^3 2^2 2 1^2 1^21} +
         3M_{3^3 222 1^41} +
          M_{3^3 2^2 2 1^41}\\
        &+ 9M_{3^3 22 2 1^31^2} +
         3M_{3^3 2^2 2 1^ 31^2}.
    \end{aligned}
    \end{equation}
\end{example}

\subsection{Expressing $P$ in Terms of $M$} In this subsection, we express the $P$-basis in terms of the monomial polysymmetric basis. This description is obtained by decomposing each component of $\tau$ through its stack divisors and combining the resulting matrix coefficients.
\begin{definition} 
Let $\tau = (d_1^{m_1}, d_2^{m_2}, \ldots, d_s^{m_s})$ be a stack partition.  
A stack divisor of $\tau$ is a tuple $k = (k_1, k_2, \ldots, k_s)$ of nonnegative integers of length $\ell(\tau)$ such that each $k_i$ divides $d_i$.  
We write $k \mid \tau$ to indicate that $k$ is a stack divisor of $\tau$.
\end{definition}
\begin{definition}
Let $\tau = (d_1^{m_1}, d_2^{m_2}, \ldots, d_s^{m_s})$ be a stack partition and let $k \mid \tau$.  
Define
\[
D(\tau,k)
  = \big( (k_1^{d_1 m_1 / k_1}),\, (k_2^{d_2 m_2 / k_2}),\, \ldots,\, (k_s^{d_s m_s / k_s}) \big)
\]
to be the sequence of stack partitions associated to $\tau$ and $k$,  
where the $i$th entry $(k_i^{d_i m_i / k_i})$ records the stack partition obtained from the $i$th entry of $\tau$.  
\end{definition}
Finally, for $\alpha\Vdash n$, let
\begin{align}
\mathbf{D}_{\tau,\alpha}
    = \sum_{k\,\mid\,\tau}\pr{\prod_{i=1}^{s} k_i}\,
\mathcal{S}^{\,\alpha}_{\,D(\tau,k)},
\end{align}
where the sum ranges over all stack divisors of $\tau$.
\begin{example}
\label{ex:2}
Let $\tau = (8,6^2,2^3,2,2)$ and $\alpha = (2^7, 1^{14},1^2)$. Among all stack divisors $k\mid\tau$, only the following contribute nontrivially to $\mathcal{S}_{D(\tau,k)}^{\alpha}$:
\begin{align}
    k^{(1)} &= (2, 1, 2, 1, 1), &
    k^{(2)} &= (1, 2, 1, 2, 1), &
    k^{(3)} &= (1, 2, 1, 1, 2).
\end{align}
For these, we have
\begin{align}
    D(\tau,k^{(1)}) &= (2^4,1^{12},2^3,1^2,1^2),\\
    D(\tau,k^{(2)}) &= (1^8,2^6,1^6,2^1,1^2),\\
    D(\tau,k^{(3)}) &= (1^8,2^6,1^6,1^2,2^1).
\end{align}
The corresponding coefficients are
\begin{align}
\mathcal{S}^{\alpha}_{D(\tau,k^{(1)})} = 2, \qquad
\mathcal{S}^{\alpha}_{D(\tau,k^{(2)})} = 1, \qquad
\mathcal{S}^{\alpha}_{D(\tau,k^{(3)})} = 1.
\end{align}
Hence
\begin{align}
\mathbf{D}_{\tau,\alpha}
    = 2(2\cdot1\cdot2\cdot1\cdot1)
     + (1\cdot2\cdot1\cdot2\cdot1)
     + (1\cdot2\cdot1\cdot1\cdot2) = 16.
\end{align}
\end{example}
\begin{theorem}
\label{thm:MxMxM..}
    Let $\tau\Vdash n$, we have
    \begin{align}
        P_{\tau} = \sum_{\alpha\Vdash n} \mathbf{D_{\tau,\alpha}}M_{\alpha}.
    \end{align} 
\end{theorem}
\begin{proof}
Let $\tau = (d_1^{m_1}, d_2^{m_2}, \ldots, d_s^{m_s})$ and $\mu_i = k_i^{d_im_i/ki}$. By definition, we have
\begin{align}
P_\tau = P_{d_1^{m_1}} \cdot P_{d_2^{m_2}} \cdots P_{d_s^{m_s}} = \left( \sum_{k_1 \mid d_1} k_1 M_{\mu_1} \right) \cdot \left( \sum_{k_2 \mid d_2} k_2 M_{\mu_2} \right) \cdots \left( \sum_{k_s \mid d_s} k_s M_{\mu_s} \right).
\end{align}
 Choosing, for each factor, a divisor $k_i$ of $d_i$, yields a stack divisor of $\tau$. Summing all stack divisors of $\tau$ gives
\begin{align}
\sum_{k\,\mid \,\tau}\pr{\prod_{i=1}^{s} k_i} \cdot \prod_{i=1}^{s} M_{\mu_i}=
\sum_{\boldsymbol{\mu}\Vdash n}
\pr{\sum_{k\,\mid \,\tau}\pr{\prod_{i=1}^{s} k_i} \mathcal{S}^{\boldsymbol{\mu}}_{D(\tau,k)}}  M_{\boldsymbol{\mu}}
= \sum_{\boldsymbol{\mu} \Vdash n} \mathbf{D}_{\tau, \boldsymbol{\mu}} \, M_{\boldsymbol{\mu}},
\end{align}
as desired.
\end{proof}

\subsection{Expressing $H, E^+$ and $E$ in terms of $M$} We now describe how the bases $H$, $E^+$, and $E$ expand in the polysymmetric monomial basis $M$.  
\begin{definition}
Let $\tau = (d_1^{m_1}, d_2^{m_2}, \ldots, d_s^{m_s}) \Vdash n$.  
For each part $\tau_i = d_i^{m_i}$, choose a stack partition $\nu_i$ of $\tau_i$.  
The sequence of stack partitions $\nu = (\nu_1, \nu_2, \ldots, \nu_s)$ is called a refinement of $\tau$.  
We say that $\nu$ is a stable refinement of $\tau$ if, for each $i$, the multiplicity of every stack in $\nu_i$ equals $m_i$.
\end{definition}
Furthermore, define 
\begin{align}
    \operatorname{sgn}_{\tau}(\nu) 
    = \prod_{i=1}^s (-1)^{\area{\nu_i}/m_i}
\end{align}
with $\operatorname{sgn}(\nu_i)$ given by \eqref{eqn:sgn}.
For $\alpha\Vdash n$, let  
\begin{align}
    \mathbf{H}_{\tau,\alpha} &= \sum_{\nu} \mathcal{S}^\alpha_{\nu},
\end{align}
where the sum is over all refinements of $\tau$. Next define
\begin{align}
    \mathbf{E}_{\tau,\alpha} = \sum_{\nu} \mathcal{S}^\alpha_{\nu},\quad
    \mathbf{ES}_{\tau,\alpha} = \sum_{\nu} \sgn{\nu}\mathcal{S}^\alpha_{\nu},
\end{align}
where the sums are over all stable refinements of $\tau$.
\begin{example}
\label{ex:3}
Let $\tau = (5, 3^2, 2)$ and $\alpha = (2^2,1^2,1^2, 1^2, 1,1, 1)$. To compute $\mathbf{H}_{\tau,\alpha}$, we list all refinements $\nu$ of $\tau$ such that $\mathcal{S}^\alpha_{\nu} > 0$.
These are:
\begin{align*}
\nu^1 = (2 1 1 1,\, 1^21^21^2,\,2), \quad \nu^2 = (11 1 1 1,\, 2^21^2,\,11), \quad \nu^3 = (2^21,\, 1^21^21^2,\,11), \\
\nu^4 = (1^2111,\, 2^21^2,\,1^2),\quad \nu^5 = (1^21^21,\, 2^21^2,\,11),
\quad \nu^6 = (1^2111,\, 2^21^2,\,11).
\end{align*}
Therefore $\mathbf{H}_{\tau,\alpha} = \sum_{i=1}^{6} \mathcal{S}^{\lambda}_{\nu^i}= 1+3+3+6+9+18=40.$ 
Now consider the stable refinements $\nu$ of $\tau$ such that $\mathcal{S}^\alpha_{\nu} > 0$.  
In this case, we obtain $\nu^1 = (2 1 1 1,\, 1^21^21^2,\,2)$ and $\nu^2 = (11 1 1 1, 2^21^2,11)$.   
Thus,
\begin{align}
\mathbf{E}_{\tau,\alpha} = \mathcal{S}^\alpha_{\nu^1}+\mathcal{S}^\alpha_{\nu^2} = 1+3=4.
\end{align}
Finally we have
\begin{align}
   \sgn{\nu^1}&= (-1)^{\area{2111}}\cdot (-1)^{\area{1^21^21^2}/2}\cdot (-1)^{\area{2}}=(1)(-1)(-1)=1,\\
   \sgn{\nu^2}&= (-1)^{\area{11111}}\cdot (-1)^{\area{2^21^2}/2}\cdot (-1)^{\area{11}}=(-1)(1)(1)=-1.
\end{align}
Then, $\mathbf{ES}_{\tau,\alpha}=\mathcal{S}^\alpha_{\nu^1}-\mathcal{S}^\alpha_{\nu^2}=-2$.
\end{example}
\begin{theorem}
    \label{thm:EandHtoMs}
    Let $\tau\Vdash n$, we have
    \begin{align}
    \label{eqn:HtoM}
    H_{\tau} &= \sum_{\alpha\Vdash n} \mathbf{H}_{\tau,\alpha}M_{\alpha},\\
     \label{eqn:EPtoM}
        E^+_{\tau} &= \sum_{\alpha\Vdash n} \mathbf{E}_{\tau,\alpha}M_{\alpha},\\
         \label{eqn:EtoM}
        E_{\tau} &= \sum_{\alpha\Vdash n} \mathbf{ES}_{\tau,\alpha}M_{\alpha}.
    \end{align}
\end{theorem}
\begin{proof}
We first give the proof for \eqref{eqn:HtoM}. 
Fix $\alpha$. Let $\tau = (\tau_1, \tau_2, \ldots, \tau_s)$. By definition,
\begin{align}
H_{\tau} = H_{\tau_1} \cdot H_{\tau_2} \cdots H_{\tau_s}.
\end{align}
To obtain a term in $H_{\tau}$, for each $\tau_i$ choose a stack partition $\nu_i$ of $\tau_i$. This choice gives the polysymmetric monomial $M_{\nu_i}$ from $H_{\tau_i}$. Multiplying the resulting $M_{\nu_i}$ over all $i$ gives the product
\begin{align}
\prod_{i=1}^sM_{\nu_i} = \sum_{\alpha\Vdash n} \mathcal{S}_{\nu}^{\alpha}M_{\alpha}.
\end{align}
Summing over all refinements $\nu$ of $\tau$ yields the total coefficient $\mathbf{H}_{\tau,\alpha}$ of $M_{\alpha}$ in $H_{\tau}$. Finally, summing over all $\alpha$ gives the expansion of $H_{\tau}$ in the polysymmetric monomial basis. The argument for \eqref{eqn:EPtoM} proceeds analogously, except that the sum ranges over the stable refinements of $\tau$, in \eqref{eqn:EtoM}, the same holds with the additional bookkeeping of signs.  
\end{proof}

\section{The Signed Complete Homogeneous and Power Polysymmetric Bases}
\label{section 4}

This section introduces two new polysymmetric bases, the signed complete homogeneous basis
$\{ H^+_{\tau} \mid \tau \Vdash n \}$ and the signed power basis
$\{ P^+_{\tau} \mid \tau \Vdash n \}$.
We first give the definitions of $H^+$ and $P^+$ and then show that
$\{ H^+_{\tau} \mid \tau \Vdash n \}$ forms a linear basis of $\PS^n$.
We also describe how $H^+_\tau$ and $P^+_\tau$ expand in the monomial polysymmetric basis.

\subsection{The new bases} The $d$\textbf{th signed  complete homogeneous polysymmetric function} is defined as
\begin{align}
    H_d^{+} = \sum_{\alpha\Vdash d} \operatorname{sgn}(\alpha)M_\alpha.
\end{align}

The $d$\textbf{th signed power polysymmetric function} is
\begin{equation} 
P^+_{d} = \sum_{k\mid d}  (-1)^{d/k}\,k\, M_{k^{d/k}},
\end{equation} 
where $k\mid d$ means $k$ divides $d$.
\begin{example}
    Let \( d = 4 \). Then the expansions in the monomial polysymmetric basis are given by
    \begin{align}
        H_4^{+} &= 
        -M_{4} 
        + M_{31} 
        + M_{22} 
        + M_{2^2} 
        - M_{211} 
        - M_{21^2} 
        + M_{1111}+ M_{1^2 11} 
        + M_{1^2 1^2}
        + M_{1^3 1}  
        + M_{1^4}, \\
        P_4^{+} &= 
        -4M_{4} + 2M_{2^2} + M_{1^4}.
    \end{align}
\end{example}
Define $H^+_{d^m}(\x{*})=H_{d}^+(\x{*}^m)\in \PS^{dm}$ and $P^+_{d^m}(\x{*})=P_{d}^+(\x{*}^m)\in \PS^{dm}$. For a stack partition $\tau = (d_1^{m_1}, d_2^{m_2},\ldots, d_s^{m_s})$, define
\begin{align}
    H_{\tau}^+ = \prod_{i=1}^{s} H_{d_i^{m_i}}^+, \quad P_{\tau}^+ = \prod_{i=1}^{s} P_{d_i^{m_i}}^+.
\end{align}
The generating functions for $H_d^+$ and $P_d^+$ are obtained by substituting
$x^m_{i,j} \mapsto -x^m_{i,j}$ into \eqref{gen:H} and \eqref{gen:P}, respectively, yielding
\begin{align}
\label{gen:HP}
    \mathcal{H}^+(t)
    &= \sum_{d=0}^{\infty} H_{d^m}^+ \, t^{md}
    = \prod_{i,j} \pr{1 + x_{i,j}^m t^{mi}}^{-1},\\
    \mathcal{P}^+(t)
    &= \sum_{d=1}^{\infty} P_{d^m}^+ \, t^{md-1}
    = \sum_{d=1}^{\infty}
      \pr{\sum_{i \mid d} i \sum_{j=1}^{\infty} (-x_{i,j}^{m})^{d/i}}
      t^{md-1}.
\end{align}

We now proceed to prove that $\{H^+_{\tau}\}$ is a linear basis of $\PS$.
\begin{theorem}{ \cite[Theorem 3.1]{g2024polysymmetricfunctionsmotivicmeasures}}
\label{thm:E,H are basis}
Each of the sets $\{E_\tau^+\}_{\tau\text{ stack partition}}$, 
$\{H_\tau\}_{\tau\text{ stack partition}}$ 
is a linear basis of $\PS$. 
Each of the sets $\{E_{d^m}^+\}_{d,m = 1}^\infty$, 
$\{H_{d^m}\}_{d, m = 1}^\infty$ 
is an algebraic basis of $\PS$. 
\end{theorem}
\begin{theorem}
    The ring endomorphism $\Omega \colon \PS \to \PS$ 
defined by $\Omega(E^+_{d^m}) = H^+_{d^m}$ 
%for all $d,m$ 
is an involution. 
\end{theorem}
\begin{proof}
     Using $\mathcal{H}^+(t)$ and $\mathcal{E}^+(t)$, the proof follows closely as presented in Proposition 3.2 of \cite{g2024polysymmetricfunctionsmotivicmeasures}.
\end{proof}
\begin{corollary}
    $\{H^+_{d^m}\}_{d, m = 1}^\infty$ is an algebraic basis of $\PS$. 
$\{H^+_\tau\}_{\tau \text{ stack partition}}$ is a linear basis of $\PS$. 
\end{corollary}
\subsection{Monomial expansions for $H^+$ and $P^+$}

Throughout this subsection, we follow the same constructions used to expand $H_\tau$ and
$P_\tau$ in Section~\ref{section 3}, applied here to their signed counterparts
$H^+_\tau$ and $P^+_\tau$ by bookkeeping the appropriate sign factors.

We first express the signed complete homogeneous basis elements $H^+_\tau$ in the monomial polysymmetric basis $M$.
Define
\begin{align}
    \mathbf{HS}_{\tau,\alpha}
    = \sum_{\nu} \operatorname{sgn}_{\tau}(\nu)\,\mathcal{S}^\alpha_{\nu},
\end{align}
where the sum is taken over all refinements $\nu$ of $\tau$.

\begin{theorem}
Let $\tau \Vdash n$. Then
\begin{align}
    H^+_{\tau} = \sum_{\alpha \Vdash n} \mathbf{HS}_{\tau,\alpha}\, M_{\alpha}.
\end{align}
\end{theorem}

\begin{example}
Continuing from Example~\ref{ex:3}, let $\tau = (5, 3^2, 2)$ and
$\alpha = (2^2, 1^2, 1^2, 1^2, 1, 1, 1)$.
As in the computation of $\mathbf{H}_{\tau,\alpha}$, the relevant refinements are the same.
Thus,
\begin{align}
    \mathbf{HS}_{\tau,\alpha}
    = \sum_{i=1}^{6} \operatorname{sgn}_{\tau}(\nu^i)\,\mathcal{S}^{\alpha}_{\nu^i}
    = 1 - 3 + 3 - 6 - 9 - 18 = -32.
\end{align}
\end{example}

We now carry out the analogous construction for the signed power basis.
As in the unsigned case, we express $P^+_\tau$ in the monomial polysymmetric basis by summing over
stack divisors of $\tau$, with the sign determined by the corresponding divisor choice.
Define
\begin{align}
\mathbf{DS}_{\tau,\alpha}
    = \sum_{k\,\mid\,\tau}
      (-1)^{\sum_{i=1}^{s} d_i/k_i}
      \left(\prod_{i=1}^{s} k_i\right)
      \mathcal{S}^{\,\alpha}_{\,D(\tau,k)},
\end{align}
where the sum ranges over all stack divisors $k \mid \tau$.

\begin{theorem}
Let $\tau \Vdash n$. Then
\begin{align}
    P^+_{\tau} = \sum_{\alpha \Vdash n} \mathbf{DS}_{\tau,\alpha}\, M_{\alpha}.
\end{align}
\end{theorem}

\begin{example}
Continuing from Example~\ref{ex:2}, let $\tau = (8,6^2,2^3,2,2)$ and
$\alpha = (2^7, 1^{14},1^2)$.
Using the same stack divisors as in the unsigned computation, we obtain
\begin{align}
\mathbf{DS}_{\tau,\alpha}
    = -2(2\cdot1\cdot2\cdot1\cdot1)
     + (1\cdot2\cdot1\cdot2\cdot1)
     + (1\cdot2\cdot1\cdot1\cdot2)
     = 0.
\end{align}
\end{example}

In this paper, every polysymmetric basis is accompanied by a signed version.
This suggests that the signing operation is not specific to individual constructions,
but is a general feature of polysymmetric bases.
It is natural to expect that future bases in this setting will admit analogous signed variants.

\section{Identities on Polysymmetric Functions}
\label{section 5}
In this section we present our main results. We give several identities between the bases of polysymmetric functions. Many of which are analogues to the classical identities for symmetric functions, but with twists that come from working in the polysymmetric setting.
\begin{remark}
    Independent work of Aditya \cite{khanna2025transitionmatricesplethysticbases},
appearing simultaneously with this paper, develops a combinatorial approach
to polysymmetric functions and derives explicit
transition formulas among 
$H$, $E$, $E^+$, and $P$.
The results of this section establish additional Jacobi--Trudi type determinant
identities, introduces a new involution, and explicit transition formulas
involving the six bases
$H, H^+, E, E^+, P,$ and $P^+$.
\end{remark}
Throughout this section, we work with the case $m = 1$, as all the proofs follow without loss of generality.  This assumption will be used without further mention and the result will be given inherently. In many of the proofs, we also apply the involution $\Omega$ defined in \cite{g2024polysymmetricfunctionsmotivicmeasures} to give the dual identities.

The following theorem shows that the polysymmetric bases $E$ and $H$ satisfy Jacobi--Trudi style determinant formulas.

\begin{theorem}
\label{thm:EdHd}
For any positive integers $d$ and $m$, 
    \begin{align}
    \label{eqn:EdHd}
        E_{d^m} = (-1)^d \det(H_{(1-i+j)^m})_{1 \leqslant i, j \leqslant d}, \\
    \label{eqn:HdEd}
        H_{d^m} = (-1)^d \det(E_{(1-i+j)^m})_{1 \leqslant i, j \leqslant d}.
    \end{align}
\end{theorem}
\begin{proof}
We aim to prove \eqref{eqn:EdHd}. Recall that $\mathcal{H}(t)\mathcal{E}(t) = 1$, so the identity
\begin{align}
\label{eqn:HdEd-i}
\sum_{i=0}^d H_i E_{d-i} = 0 \quad \text{for } d \ge 1
\end{align}
follows by equating coefficients of $t^d$. We proceed by induction on $d$. For the base case $d = 1$, equation~\eqref{eqn:HdEd-i} gives $H_1 E_0 + H_0 E_1 = 0$, so
\begin{align}
E_1 = -H_1 = -\det(H_1),
\end{align}
as desired. Assume the formula holds for all positive integers $k<d$.  
Consider the $d\times d$ matrix
\begin{align}
A = \begin{bmatrix}
H_1 & H_2 & H_3 & \cdots & H_d \\
H_0 & H_1 & H_2 & \cdots & H_{d-1} \\
0   & H_0 & H_1 & \cdots & H_{d-2} \\
\vdots & \vdots & \vdots & \ddots & \vdots \\
0 & 0 & 0 & \cdots & H_1
\end{bmatrix},
\end{align}
whose determinant is $\det(H_{1-i+j})_{1 \le i,j \le d}$. Let $A_{i,j}$ denote the submatrix obtained by deleting the $i$th row and $j$th column of $A$. Expanding along the first column, we obtain
\begin{align}
\det A = H_1 \det A_{1,1} - H_0 \det A_{2,1}.
\end{align}
By the inductive hypothesis, $\det A_{1,1} = (-1)^{d-1} E_{d-1}$, and since $H_0 = 1$, we have
\begin{align}
\det A = (-1)^{d-1} H_1 E_{d-1} - \det A_{2,1}.
\end{align}
We expand $\det A_{2,1}$ in the same manner, applying the inductive hypothesis to each resulting minor. Continuing recursively through all lower order subdeterminants, we obtain
\begin{align}
\det A = (-1)^{d-1} (H_1 E_{d-1} + H_2 E_{d-2} + \cdots + H_d E_0).
\end{align}
Thus,
\begin{align}
\label{eqn:inductivedeterH}
(-1)^{d-1} \det(H_{1-i+j})_{1 \le i,j \le d} = \sum_{i=1}^d H_i E_{d-i}.
\end{align}
Substituting \eqref{eqn:HdEd-i} into this expression gives
\begin{align}
E_d = -\sum_{i=1}^d H_i E_{d-i} = (-1)^d \det(H_{1-i+j})_{1 \le i,j \le d},
\end{align}
as claimed. Applying the involution $\Omega$ yields \eqref{eqn:HdEd}.
\end{proof}
The Jacobi--Trudi style determinant formula also holds for the signed versions of the polysymmetric bases $E$ and $H$.

\begin{theorem}
\label{thm:PlusToPlus}
For any positive integers $d$ and $m$, we have
\begin{align}
    \label{eqn:EPinHP}
    E^+_{d^m} &= (-1)^d \det\left(H^+_{(1 - i + j)^m}\right)_{1 \leqslant i, j \leqslant d}, \\
    \label{eqn:HPinEP}
    H^+_{d^m} &= (-1)^d \det\left(E^+_{(1 - i + j)^m}\right)_{1 \leqslant i, j \leqslant d}.
\end{align}
\end{theorem}

\begin{proof}
As in Theorem~\ref{thm:EdHd}, equation~\eqref{eqn:EPinHP} follows by the same induction argument, using
$\mathcal{H}^+(t)\mathcal{E}^+(t)=1$ and replacing $H$ and $E$ with $H^+$ and $E^+$, respectively, throughout.
Applying the involution $\Omega$ then yields \eqref{eqn:HPinEP}.
\end{proof}
The next theorem provides essential identities that complete the relationships among the four bases $H, E, H^+,$ and $E^+$.
\begin{theorem}
For any positive integers $d$ and $m$, we have
\begin{align}
    \label{eqn:H2EPtoH}
    H_{d^m} &= \sum_{k=0}^{\left\lfloor \frac{d}{2} \right\rfloor} H_{k^{2m}} E^+_{(d - 2k)^m}, \\
    \label{eqn:H2EtoHP}
    H^+_{d^m} &= \sum_{k=0}^{\left\lfloor \frac{d}{2} \right\rfloor} H_{k^{2m}} E_{(d - 2k)^m}, \\
    \label{eqn:E2HPtoE}
    E_{d^m} &= \sum_{k=0}^{\left\lfloor \frac{d}{2} \right\rfloor} E_{k^{2m}} H^+_{(d - 2k)^m}, \\
    \label{eqn:E2HtoEP}
    E^+_{d^m} &= \sum_{k=0}^{\left\lfloor \frac{d}{2} \right\rfloor} E_{k^{2m}} H_{(d - 2k)^m}.
\end{align}
\end{theorem}

\begin{proof}
We begin with the generating function
\begin{align}
    \sum_d H_{d^2} t^{2d} &= \prod_{i,j} (1 - x_{i,j}^{2} t^{2i})^{-1} = \prod_{i,j} (1 - x_{i,j} t^i)^{-1} (1 + x_{i,j} t^i)^{-1}.
\end{align}
Multiplying both sides by \( \mathcal{E}^+(t) \), we obtain
\begin{align}
    \left( \sum_d H_{d^2} t^{2d} \right) \left( \sum_d E^+_d t^d \right) 
    &= \prod_{i,j} (1 - x_{i,j} t^i)^{-1}.
\end{align}
Thus, we have the identity
\begin{align}
\label{eqn:H2EPH}
    \left( \sum_d H_{d^2} t^{2d} \right) \left( \sum_d E^+_d t^d \right) = \sum_d H_d t^d.
\end{align}
Equating coefficients of \( t^d \) gives equation~\eqref{eqn:H2EPtoH}. Replacing \( \mathcal{E}^+(t) \) with \( \mathcal{E}(t) \) yields equation~\eqref{eqn:H2EtoHP}. Finally, applying the involution $\Omega$ to equations~\eqref{eqn:H2EPtoH} and~\eqref{eqn:H2EtoHP} yields equations~\eqref{eqn:E2HPtoE} and~\eqref{eqn:E2HtoEP}, respectively.
\end{proof}

Given a partition $\lambda$ and a stack partition $\tau$, let $m_i(\lambda)$ denote the number of times the part $i$ appears in $\lambda$, and let $m_{i,j}(\tau)$ denote the number of times the term $i^j$ appears in $\tau$.
 For each $j$, define
\begin{align}
    N_j(\tau) = \sum_i m_{i,j}(\tau),
\end{align}
the total number of terms in $\tau$ whose multiplicity is $j$. Then define
\begin{align}
    \kappa_\tau = \prod_j \binom{N_j(\tau)}{m_{1,j}(\tau),\, m_{2,j}(\tau),\, \ldots}.
\end{align}

\begin{example}
Let
$\tau = (3^3,  3^3,2^3,2^3,2^3,2^3, 2^2, 2^2, 1^3,1^2, 1^1, 1^1).
$
The nonzero values of $m_{i,j}(\tau)$ are
\begin{align*}
m_{3,3}(\tau) &= 2, &
m_{2,3}(\tau) &= 4, &
m_{2, 2}(\tau) &= 2, \\
m_{1,3}(\tau) &= 1, &
m_{1,2}(\tau) &= 1, &
m_{1,1}(\tau) &= 2.
\end{align*}
Hence
\begin{align*}
N_3(\tau) &= m_{1,3}(\tau)+m_{2,3}(\tau)+m_{3,3}(\tau) = 7, \\
N_2(\tau) &= m_{1,2}(\tau)+m_{2,2}(\tau) = 3, \\
N_1(\tau) &= m_{1,1}(\tau) = 2.
\end{align*}
It follows that $\kappa_\tau
=
\binom{7}{1,4,2}
\binom{3}{1,2}
\binom{2}{2}
= 105$.
\end{example}

For a stack partition $\tau$, we write $\tau^m$ for the stack partition obtained by multiplying
each multiplicity in $\tau$ by $m$.

\begin{example}
If $\tau = (3^5,2^3,1^1)$, then $\tau^4 = (3^{20},2^{12},1^4)$.
\end{example}

We are now ready to use $\kappa_\tau$ to express the signed bases
$E^+_{d^m}$ and $H^+_{d^m}$ in terms of the unsigned bases
$E_{d^m}$ and $H_{d^m}$.

\begin{theorem}
\label{thm:BPtoB}
Given positive integers $d$ and $m$, we have
\begin{align}
\label{eqn:EPinEs}
E^+_{d^m} &= \sum_{\tau} \sgn{\tau}\, \kappa_\tau\, E_{\tau^m}, \\
\label{eqn:HPinHs}
H^+_{d^m} &= \sum_{\tau} \sgn{\tau}\, \kappa_\tau\, H_{\tau^m},
\end{align}
where the sum is over all stack partitions $\tau$ of $d$ such that each term of $\tau$ has multiplicity either $1$ or $2$, and at most one term has multiplicity $2$.
\end{theorem}

\begin{proof}
Expanding the determinant in \eqref{eqn:HdEd} yields
\begin{align}
\det(E_{1-i+j})_{1 \le i,j \le d}
=
\sum_{\lambda\vdash d}
(-1)^{d-\ell(\lambda)}
\binom{\ell(\lambda)}{m_{1}(\lambda),\, m_{2}(\lambda),\, \ldots}\,
E_\lambda,
\end{align}

Substituting this expansion into \eqref{eqn:E2HtoEP} gives
\begin{align*}
E^+_{d}
=
\sum_{k = 0}^{\left\lfloor d/2 \right\rfloor}
E_{k^{2}}
\left(
\sum_{\lambda\vdash d - 2k}
(-1)^{\ell(\lambda)}
\binom{\ell(\lambda)}{m_{1}(\lambda),\, m_{2}(\lambda),\, \ldots}\,
E_\lambda
\right).
\end{align*}
Expanding the outer sum produces terms of the form $E_\tau$, where
$\tau \Vdash d$ is a stack partition in which each term has multiplicity
either $1$ or $2$, and at most one term has multiplicity $2$, exactly the
stack partitions appearing in \eqref{eqn:EPinEs}.
Applying the involution $\Omega$ to both sides yields \eqref{eqn:HPinHs}.
\end{proof}

A similar result follows when expressing $E^+$ and $H^+$ in terms of $H$ and $E$, respectively.

\begin{theorem}
\label{thm:EPtoH}
Given a positive integer $d$, we have
\begin{align}
\label{eqn:EPinHs}
E^+_{d^m} &= \sum_{\tau} (-1)^{N_2(\tau)} \kappa_\tau H_{\tau^m}, \\
\label{eqn:HPinEs}
H^+_{d^m} &= \sum_{\tau} (-1)^{N_2(\tau)} \kappa_\tau E_{\tau^m},
\end{align}
where the sum is over all stack partitions $\tau \Vdash d$ in which each term has multiplicity in $\{1, 2\}$ and at most one term appears with multiplicity $1$. 
\end{theorem}

\begin{proof}
The proof proceeds analogously to Theorem \ref{thm:BPtoB}, using the determinant identity \eqref{eqn:EdHd} and the relation \eqref{eqn:E2HtoEP}.
\end{proof}
We now present the inverse relation, expressing $E$ and $H$ in terms of $E^+$ and $H^+$.
\begin{theorem}
\label{thm:EtoHP}
For any positive integer $d$, we have
\begin{align}
\label{eqn:HinEP}
H_{d^m} &= \sum_{\tau} E^+_{\tau^m}, \\
\label{eqn:EinHP}
E_{d^m} &= \sum_{\tau} H^+_{\tau^m},
\end{align}
where the sum is over all stack partitions $\tau$ of $d$ such that each term appears
with a distinct power of $2$ as its multiplicity.
\end{theorem}

\begin{proof}
It is enough to prove \eqref{eqn:HinEP}, since \eqref{eqn:EinHP} then follows by
applying the involution $\Omega$. Write $H_d$ in the monomial polysymmetric basis as
\begin{align}
H_d = \sum_{\alpha}  M_\alpha,
\end{align}
where the sum is over stack partitions $\alpha = (d_1^{m_1}, d_2^{m_2}, \ldots, d_s^{m_s})$ of $d$. Fix such an index $\alpha$. For each $i$, write the multiplicity $m_i$ in binary as a sum of distinct powers of $2$,
\begin{align}
m_i = 2^{i_1} + 2^{i_2} + \cdots + 2^{i_{k_i}},
\qquad i_1 > i_2 > \cdots > i_{k_i} \ge 0.
\end{align}
Replace the stack $d_i^{m_i}$ by the collection
\begin{align}
d_i^{2^{i_1}},\; d_i^{2^{i_2}},\; \ldots,\; d_i^{2^{i_{k_i}}}.
\end{align}

Perform this replacement for each $i$, and then group together all stacks having the
same multiplicity. That is, for each power $2^k$, group together all stacks of the form
$d_j^{2^k}$ that appear. This group contributes a factor
\begin{align}
E^+_{n_k^{2^k}},
\end{align}
where $n_k$ is the sum of the degrees $d_j$ of the stacks in this group.
Repeating this process for all $k$, we obtain a product
\begin{align}
E^+_{\mu} = \prod_k E^+_{n_k^{2^k}},
\end{align}
where $\mu$ is the stack partition formed by reassembling these groups. Each monomial polysymmetric basis element $M_\alpha$ appears in exactly one such
$E^+_\mu$ due to the uniqueness of binary decomposition. Since each $M_\alpha$
appears in exactly one $E^+_\mu$, the sum over all such $E^+_\mu$ recovers all
monomial polysymmetric basis elements of $H_d$ exactly once. This shows that
$H_d$ is equal to the sum of $E^+_\mu$ over all stack partitions $\mu$ with distinct
powers of $2$ as multiplicities, proving \eqref{eqn:HinEP}.
\end{proof}

We next give expansions of $E$ and $H$ in terms of $E^+$ and $H^+$.

\begin{theorem}
\label{thm:BtoB+}
For any positive integer $d$, we have
\begin{align}
\label{eqn:EinEPs}
E_{d^m} &= \sum_{\tau}  (-1)^{\ell(\tau)}\kappa_\tau E^+_{\tau^m}, \\
\label{eqn:HinHPs}
H_{d^m} &= \sum_{\tau}  (-1)^{\ell(\tau)}\kappa_\tau H^+_{\tau^m},
\end{align}
where the sums are over all stack partitions $\tau = (d_1^{m_1}, d_2^{m_2}, \dots, d_s^{m_s})$ of $d$, where each $m_i$ is a power of $2$.
\end{theorem}

\begin{proof}
We start from \eqref{eqn:EinHP}, which gives
\begin{align}
E_{d} = \sum_{\tau} H^+_{\tau},
\end{align}
where the sum is over all stack partitions $\tau$ of $d$ in which each multiplicity is a distinct power of $2$.
Write such a stack partition as $\tau = (d_1^{m_1}, d_2^{m_2}, \ldots, d_s^{m_s}),$
so that
\begin{align}
H^+_{\tau} = H^+_{d_1^{m_1}} \cdot H^+_{d_2^{m_2}} \cdots H^+_{d_s^{m_s}}.
\end{align}
For each factor $H^+_{d_k^{m_k}}$, apply the determinant expansion \eqref{eqn:HPinEP}, yielding
\begin{align}
E_{d}
= \sum_{\tau} \prod_{k=1}^{s}
\left(
\sum_{\lambda_k \vdash d_k}
(-1)^{\ell(\lambda_k)}
\binom{\ell(\lambda_k)}{m_{1}(\lambda_k),\, m_{2}(\lambda_k),\, \dots}
\,E^+_{\lambda_k^{m_k}}
\right).
\end{align}
Now fix $\tau=(d_1^{m_1},\ldots,d_s^{m_s})$ and choose partitions $\lambda_k \vdash d_k$ for each $k$.
Since the multiplicities $m_1,\ldots,m_s$ are distinct powers of $2$, the product
\begin{align}
E^+_{\lambda_1^{m_1}}\cdot E^+_{\lambda_2^{m_2}}\cdots E^+_{\lambda_s^{m_s}}
\end{align}
is a single basis element $E^+_{\mu}$, where $\mu$ is obtained by concatenating the parts of the $\lambda_k$
and assigning multiplicity $m_k$ to every part coming from $\lambda_k$.
Equivalently, $\mu$ is a stack partition in which each multiplicity is a power of $2$.
Collecting all such terms over all choices of $\tau$ and $\lambda_k$ gives \eqref{eqn:EinEPs}.
Applying the involution $\Omega$ to both sides yields \eqref{eqn:HinHPs}.
\end{proof}

To understand the action of the involution $\Omega$ on $P_\tau$, we examine the generating function $\mathcal{P}(t)$ and its relationships with $\mathcal{H}(t)$ and $\mathcal{E}(t)$. The identities that follow provide the necessary groundwork.

\begin{theorem}
\label{thm:dHd}
For any positive integers $d$ and $m$, we have
\begin{align}
\label{eqn:HPH}
    H_{d^m} = \frac{1}{d}\sum_{i=1}^d P_{i^m} H_{(d-i)^m}.
\end{align}
\end{theorem}

\begin{proof}
Consider the generating function for $P_d$. Letting $d = ir$, we obtain
\begin{align}
    \mathcal{P}(t)
    = \sum_{d=1}^\infty \left( \sum_{i \mid d} i \sum_{j=1} ^\infty x_{i,j}^{d/i} \right) t^{d-1}
    &= \sum_{i=1}^\infty  \sum_{j=1} ^\infty \sum_{r=1}^\infty i x_{i,j}^r t^{ir - 1} \\
    &= \sum_{i=1}^\infty \sum_{j=1} ^\infty  i x_{i,j} t^{i - 1} \sum_{r=0}^\infty (x_{i,j} t^i)^r \\
    &= \sum_{i=1}^\infty \sum_{j=1}^\infty  \frac{d}{dt} \log\left( \frac{1}{1 - x_{i,j} t^i} \right).
\end{align}
Therefore,
\begin{align}
\label{eqn:GenP}
    \mathcal{P}(t)
    = \frac{d}{dt} \log \prod_{i,j} \frac{1}{1 - x_{i,j} t^i}
    = \frac{d}{dt} \log \mathcal{H}(t)
    = \frac{\mathcal{H}'(t)}{\mathcal{H}(t)}.
\end{align}
Multiplying both sides by $\mathcal{H}(t)$ yields
\begin{align}
    \mathcal{P}(t) \cdot \mathcal{H}(t) = \mathcal{H}'(t).
\end{align}
Comparing coefficients of $t^{d-1}$ on both sides gives our result.
\end{proof}
\begin{remark}
Solving for $\mathcal{H}(t)$ in equation \eqref{eqn:GenP} gives
\begin{align}
\label{eqn:GenH=GenP}
\sum_{d=0}^{\infty} H_d t^d
= \exp\left( \sum_{k=1}^{\infty} \frac{P_k t^k}{k} \right)
= \left( \sum_{d=0}^{\infty} E_d t^d \right)^{-1}.
\end{align}
\end{remark}

\begin{theorem}
\label{thm:dEd}
For any positive integers $d$ and $m$, we have
\begin{align}
E_{d^m} = -\frac{1}{d}\sum_{i=1}^d P_{i^m} \, E_{(d-i)^m}.
\end{align}
\end{theorem}

\begin{proof}
This follows from equation \eqref{eqn:GenP} by substituting
$\mathcal{H}(t)=\mathcal{E}(t)^{-1}$ and equating coefficients of $t^{d-1}$,
exactly as in the proof of Theorem~\ref{thm:dHd}.
\end{proof}

\begin{remark}
Replacing $\mathcal{H}(t)$ with $\mathcal{E}(t)^{-1}$ in equation \eqref{eqn:GenP}
and equating coefficients of $t^{d-1}$ gives
\begin{align}
P_{d^m} = \sum_{i=0}^{d-1} (d-i)\, E_{i^m}\, H_{(d-i)^m}.
\end{align}
\end{remark}

We now describe how the involution $\Omega$ acts on the power polysymmetric functions.

\begin{theorem}
For any stack partition $\tau = (d_1^{m_1}, d_2^{m_2}, \ldots, d_s^{m_s})$, we have
\begin{align}
\Omega(P_\tau) = (-1)^{\ell(\tau)} P_\tau.
\end{align}
\end{theorem}
\begin{proof}
The involution $\Omega$ interchanges $\mathcal{E}(t)$ and $\mathcal{H}(t)$. Since
\begin{align}
\mathcal{P}(t) = \frac{d}{dt} \log \mathcal{H}(t)
\quad \text{and} \quad
\mathcal{P}(t) = -\frac{d}{dt} \log \mathcal{E}(t),
\end{align}
it follows that $\Omega(P_{d^m}) = -P_{d^m}$ for all positive integers $d$ and $m$.
Hence, for any stack partition $\tau$, we have that $\Omega(P_\tau) = (-1)^{\ell(\tau)} P_\tau$.
\end{proof}

We now give an explicit expansion of the power polysymmetric function $P$ in terms of $H$ and $E$, respectively.

\begin{theorem}
\label{thm:PtoH}
    For any positive integers $d$ and $m$, we have
    \begin{align}
    \label{eqn:PtoH}
        P_{d^m} &= d\sum_{\lambda\vdash d}(-1)^{\ell(\lambda)-1}
        \frac{1}{\ell(\lambda)}
  \binom{\ell(\lambda)}{m_1(\lambda), m_2(\lambda), \ldots}
        H_{\lambda^m},
\\
\label{eqn:PtoE}
  P_{d^m} &= d\sum_{\lambda\vdash d}(-1)^{\ell(\lambda)}
  \frac{1}{\ell(\lambda)}
  \binom{\ell(\lambda)}{m_1(\lambda), m_2(\lambda), \ldots}
  E_{\lambda^m},
    \end{align}
\end{theorem}
\begin{proof}
By \eqref{eqn:GenH=GenP}, taking the logarithm of both sides gives
\begin{align}
\sum_{k=1}^{\infty}\frac{P_k}{k}t^k
&=
\log\left(1+\sum_{n=1}^{\infty} H_n t^n\right) =
\sum_{\ell=1}^{\infty}\frac{(-1)^{\ell-1}}{\ell}
\left(\sum_{n=1}^{\infty} H_n t^n\right)^{\ell}.
\end{align}
Taking the coefficient of $t^d$ on both sides, the left-hand side contributes
$\frac{P_d}{d}$, while the right-hand side expands as
\begin{align}
\left[t^d\right]
\sum_{\ell=1}^{\infty}\frac{(-1)^{\ell-1}}{\ell}
\left(\sum_{n=1}^{\infty} H_n t^n\right)^{\ell}
&=
\sum_{\lambda\vdash d}
(-1)^{\ell(\lambda)-1}\,
\frac{1}{\ell(\lambda)}\,
\frac{\ell(\lambda)!}{\prod_i m_i(\lambda)!}\,
H_{\lambda} \\
&=
\sum_{\lambda\vdash d}
(-1)^{\ell(\lambda)-1}\,
\frac{(\ell(\lambda)-1)!}{\prod_i m_i(\lambda)!}\,
H_{\lambda}.
\end{align}
Multiplying both sides by $d$ yields \eqref{eqn:PtoH}.
Applying the involution $\Omega$ gives \eqref{eqn:PtoE}.
\end{proof}
The expansion in Theorem~\ref{thm:PtoH} can be rewritten equivalently
as a Jacobi--Trudi type determinant.
\begin{corollary}
\label{cor:Pdet}
For any positive integers $d$ and $m$, we have
\begin{align}
(-1)^{d-1}P_{d^m}
&=
\det\!\begin{pmatrix}
H_{1^m} & 1 & 0 & \cdots & 0 \\
2H_{2^m} & H_{1^m} & 1 & \cdots & 0 \\
\vdots & \vdots & \ddots & \ddots & 1 \\
dH_{d^m} & H_{(d-1)^m} & \cdots & H_{1^m}
\end{pmatrix},
\\
(-1)^{d}P_{d^m}
&=
\det\!\begin{pmatrix}
E_{1^m} & 1 & 0 & \cdots & 0 \\
2E_{2^m} & E_{1^m} & 1 & \cdots & 0 \\
\vdots & \vdots & \ddots & \ddots & 1 \\
dE_{d^m} & E_{(d-1)^m} & \cdots & E_{1^m}
\end{pmatrix}.
\end{align}
\end{corollary}

\begin{definition}
Following standard notation in symmetric function theory, for any partition
$\lambda$, define
\begin{align}
z_\lambda
=
\prod_{i\ge 1} i^{m_i(\lambda)}\, m_i(\lambda)!.
\end{align}
\end{definition}
\begin{example}
Let $\lambda = (4,4,2,1,1,1)$. Then
\begin{align}
m_1(\lambda)=3,\quad m_2(\lambda)=1,\quad m_4(\lambda)=2,
\end{align}
and hence $z_\lambda
= 1^{3}\cdot 3!\cdot 2^{1}\cdot 1!\cdot 4^{2}\cdot 2!
= 384.$
\end{example}

We now express the homogeneous and elementary polysymmetric functions in terms of the power polysymmetric basis.

\begin{theorem}
\label{thm:HandP}
For positive integers $d$ and $m$, we have
\begin{align}
\label{eqn:HandP}
    H_{d^m} &= \sum_{\lambda\vdash d} \frac{1}{z_\lambda} P_{\lambda^m}, \\
\label{eqn:EandP}
    E_{d^m} &= \sum_{\lambda\vdash d} \frac{(-1)^{\ell(\lambda)}}{z_\lambda} P_{\lambda^m}.
\end{align}
\end{theorem}

\begin{proof}
We prove \eqref{eqn:HandP}; the identity \eqref{eqn:EandP} follows by applying the involution $\Omega$ to both sides.
Consider the generating function for the homogeneous polysymmetric functions
\begin{align}
\mathcal{H}(t)
    = \sum_{d=0}^{\infty} H_d t^d
    = \exp\left( \sum_{d=1}^{\infty} P_d \frac{t^d}{d} \right).
\end{align}
This can be rewritten as
\begin{align}
\mathcal{H}(t)
    &= \prod_{d=1}^{\infty} \exp\left( \frac{P_d t^d}{d} \right) \\
    &= \prod_{d=1}^{\infty} \sum_{k_d = 0}^{\infty} \frac{(P_d t^d)^{k_d}}{d^{k_d} k_d!} \\
    &= \sum_{\lambda} \frac{1}{z_\lambda} P_\lambda t^{|\lambda|} \\
    &= \sum_{d=0}^{\infty} \sum_{\lambda \vdash d} \frac{1}{z_\lambda} P_\lambda t^d.
\end{align}
Comparing coefficients of $t^d$ yields \eqref{eqn:HandP}.
\end{proof}
\begin{definition}
Let $\lambda$ and $\mu$ be integer partitions.
Their \emph{union}, denoted $\lambda \cup \mu$, is the partition obtained by
taking the multiset union of the parts of $\lambda$ and $\mu$ and reordering
the result into weakly decreasing order. In particular, for a positive integer $i$, writing $\lambda \cup i$ means
$\lambda \cup (i)$.
\end{definition}
\begin{example}
If $\lambda = (4,2,1)$ and $\mu = (3,2)$, then $\lambda \cup \mu = (4,3,2,2,1).$
\end{example}

Lastly, the transition from $P$ to $E^+$ and $H^+$ yields expansions analogous to those relating $P$ with $E$ and $H$.

\begin{theorem}
\label{thm:PlustoP}
Given any positive integers $d$ and $m$, we have
\begin{align}
\label{eqn:EPtoP}
E_{d^m}^+ &= \sum_{\lambda,\mu} z_\lambda^{-1} z_\mu^{-1} (-1)^{\ell(\mu)} P_{\lambda^m} P_{\mu^{2m}}, \\
\label{eqn:HPtoP}
H_{d^m}^+ &= \sum_{\lambda,\mu} z_\lambda^{-1} z_\mu^{-1} (-1)^{\ell(\lambda)} P_{\lambda^m} P_{\mu^{2m}},
\end{align}
where $\lambda$ and $\mu$ range over partitions such that $|\lambda| + 2|\mu| = d$.
\end{theorem}

\begin{proof}
We prove \eqref{eqn:EPtoP},  the identity \eqref{eqn:HPtoP} follows by applying $\Omega$.
Consider the generating function
\begin{align}
\sum_{d=0}^\infty H_{d^2} t^{2d}
= \exp\left( \sum_{k=1}^\infty P_{k^2} \frac{t^{2k}}{k} \right).
\end{align}
By \eqref{eqn:H2EPH} and \eqref{eqn:GenH=GenP}, we have
\begin{align}
\exp\left( \sum_{h=1}^\infty P_{h^2} \frac{t^{2h}}{h} \right)
\sum_{d=0}^\infty E_d^+ t^d
= \exp\left( \sum_{k=1}^\infty P_k \frac{t^k}{k} \right).
\end{align}
Solving for the generating function of $E_d^+$ yields
\begin{align}
\label{eqn:EPP^2}
\sum_{d=0}^\infty E_d^+ t^d
= \exp\left( \sum_{k=1}^\infty P_k \frac{t^k}{k} \right)
\exp\left( - \sum_{h=1}^\infty P_{h^2} \frac{t^{2h}}{h} \right).
\end{align}
Expanding both exponentials gives
\begin{align}
\sum_{d=0}^\infty E_d^+ t^d
&= \left( \sum_{\lambda} z_\lambda^{-1} P_{\lambda^m} t^{|\lambda|} \right)
   \left( \sum_{\mu} z_\mu^{-1} (-1)^{\ell(\mu)} P_{\mu^{2m}} t^{2|\mu|} \right) \\
&= \sum_{d=0}^\infty
   \sum_{\substack{\lambda,\mu }}
   z_\lambda^{-1} z_\mu^{-1} (-1)^{\ell(\mu)} P_{\lambda^m} P_{\mu^{2m}} t^d.
\end{align}
where the inner sum is over all $\lambda,\mu$ with $|\lambda|+2|\mu|=d$. Comparing coefficients of $t^d$ yields \eqref{eqn:EPtoP}.
\end{proof}

\begin{theorem}
\label{thm:PtoPlus}
Let $d$ and $m$ be positive integers. Then
\begin{align}
\label{eqn:PEP}
P_{d^m}
&=
d\sum_{r \mid d}
\ \sum_{\lambda \vdash d/r}
(-1)^{\ell(\lambda)-1}
\frac{1}{\ell(\lambda)}
\binom{\ell(\lambda)}{m_{1}(\lambda),\,m_{2}(\lambda),\,\ldots}
E^+_{\lambda^{rm}},
\\
\label{eqn:PHP}
P_{d^m}
&= d
\sum_{r \mid d}
\ \sum_{\lambda \vdash d/r}
(-1)^{\ell(\lambda)}
\frac{1}{\ell(\lambda)}
\binom{\ell(\lambda)}{m_{1}(\lambda),\,m_{2}(\lambda),\,\ldots}
H^+_{\lambda^{rm}},
\end{align}
where in both sums $r$ ranges over all divisors of $d$ that are powers of $2$.
\end{theorem}
\begin{proof}
It suffices to prove \eqref{eqn:PEP}, since \eqref{eqn:PHP} follows by applying the involution $\Omega$.
Taking the logarithm of both sides of \eqref{eqn:EPP^2} gives
\begin{align}
\sum_{d=1}^\infty P_d \frac{t^d}{d}
-
\sum_{d=1}^\infty P_{d^2} \frac{t^{2d}}{d}
=
\log(\mathcal E^+(t)).
\end{align}
   Differentiating with respect to $t$ and applying
\eqref{eqn:HPinEP} to expand $H^+$ in terms of $E^+$, we obtain

\begin{align}
\sum_{d=1}^\infty P_d t^{d-1}
-
\sum_{d=1}^\infty 2 P_{d^2} t^{2d-1}
&=
(\mathcal E^+(t))' \mathcal H^+(t)
\\
&=
\left( \sum_{i=1}^\infty i E_i^+ t^{i-1} \right)
\left(
\sum_{j=0}^\infty
\left(
\sum_{|\lambda|=j}
(-1)^{\ell(\lambda)}
\binom{\ell(\lambda)}{m_{1}(\lambda),\,m_{2}(\lambda),\,\ldots}
E_\lambda^+
\right)
t^j
\right).
\end{align}

We now compare coefficients of $t^{d-1}$.
Writing $\nu = \lambda \cup i$, where $i>0$ and $|\lambda|+i=d$, we have
$\ell(\nu)=\ell(\lambda)+1$, and hence
$(-1)^{\ell(\lambda)} = (-1)^{\ell(\nu)-1}$.
This gives
\begin{align}
\label{eqn:PdPiecewiseClean}
P_d
=
\sum_{i+j=d}
\sum_{|\lambda|=j}
(-1)^{\ell(\nu)-1}
i
\binom{\ell(\lambda)}{m_{1}(\lambda),\,m_{2}(\lambda),\,\ldots}
E_\nu^+
+
\begin{cases}
2P_{(d/2)^{2}}, & \text{if $d$ is even},\\
0, & \text{if $d$ is odd}.
\end{cases}
\end{align}
Here, the additional term $2P_{(d/2)^2}$ arises from the second summation on the left-hand side. Moreover,
\begin{align}
i
\binom{\ell(\lambda)}{m_{1}(\lambda),\,\ldots}
=
\frac{i}{\ell(\nu)}
\binom{\ell(\nu)}{m_{1}(\nu),\,\ldots,m_{i}(\nu)-1,\ldots}
=
\frac{i m_{i}(\nu)}{\ell(\nu)}
\binom{\ell(\nu)}{m_{1}(\nu),\,\ldots}.
\end{align}
Summing over all $i$, and using the identity
$d=\sum_i i m_{i}(\nu)$, we obtain
\begin{align}
\sum_i
\frac{i m_{i}(\nu)}{\ell(\nu)}
\binom{\ell(\nu)}{m_{1}(\nu),\,\ldots}
=
\frac{d}{\ell(\nu)}
\binom{\ell(\nu)}{m_{1}(\nu),\,\ldots}.
\end{align}

Putting everything together gives
\begin{align}
\label{eqn:PdPiecewise}
P_d
&=
d\sum_{\nu \vdash d}
(-1)^{\ell(\nu)-1}
\frac{1}{\ell(\nu)}
\binom{\ell(\nu)}{m_{1}(\nu),\,m_{2}(\nu),\,\ldots}
E_\nu^+
+
\begin{cases}
2P_{(d/2)^{2}}, & \text{if $d$ is even},\\
0, & \text{if $d$ is odd}.
\end{cases}
\end{align}

When $d$ is odd, the additional term does not appear, so only the contribution with $r=1$ occurs, yielding the desired formula.
When $d$ is even, the term $2P_{(d/2)^2}$ contributes, corresponding to the case $r=2$.
Applying the same expansion recursively produces contributions with $r=4,8,\ldots$, continuing through all powers of $2$ dividing $d$.
Thus, the expansion yields precisely the sum over all divisors $r\mid d$ that are powers of $2$, as stated in Theorem~\ref{thm:PtoPlus}.
The factor $2$ at each stage ensures that the coefficient in front remains $d$, completing the proof.
\end{proof}

We define the involution $\Omega^+$ by $\Omega^+(x_{i,j})=-x_{i,j}$.
Although many signed identities are formally related to their unsigned
counterparts via $\Omega^+$, the generating function proofs from
$P$ to $E$ and $P$ to $H$ do not transfer directly.
Indeed, those proofs rely on multiplicative factorizations of the form
$(1-x_{i,j}t)^{-1}=(1-x_{i,j}^2t^2)^{-1}(1+x_{i,j}t)$, which are not preserved
after applying $\Omega^+$.

For this reason, while we state several signed analogues here, we provide
separate proofs of the $P^+$–$E$ and $P^+$–$H$ expansions

\begin{theorem}
\label{thm:PP}
    For positive integers $d$ and $m$, we have
   \begin{align}
   \Omega(P^+_\tau) &= (-1)^{\ell(\tau)} P^+_\tau.\\
\label{eqn:HPtoPPHP}
        H^+_{d^m} &= \frac{1}{d}\sum_{i=1}^d P^+_{i^m} H^+_{(d-i)^m}.\\
\label{eqn:PPtoEPHP}
    P^+_{d^m} &= \sum_{i=0}^{d-1} (d - i) E^+_{i^m} H^+_{(d - i)^m}.\\
\label{eqn:EPtoPPEP}
     E^+_{d^m} &= -\frac{1}{d}\sum_{i=1}^d P^+_{i^m} \, E^+_{(d - i)^m}.\\
\label{eqn:PPtoHP}
        P^+_{d^m} &= \sum_{\lambda\vdash d}(-1)^{\ell(\lambda)-1}
        \frac{d}{\ell(\lambda)}
  \binom{\ell(\lambda)}{m_1(\lambda), m_2(\lambda), \ldots}
        H^+_{\lambda^m},
\\
\label{eqn:PPtoEP}
  P^+_{d^m} &= \sum_{\lambda\vdash d}(-1)^{\ell(\lambda)}
  \frac{d}{\ell(\lambda)}
  \binom{\ell(\lambda)}{m_1(\lambda), m_2(\lambda), \ldots}
  E^+_{\lambda^m},\\
\label{eqn:HPtoPP}
    H^+_{d^m} &= \sum_{\lambda\vdash d} \frac{1}{z_\lambda} P^+_{\lambda^m}, \\
\label{eqn:EPtoPP}
    E^+_{d^m} &= \sum_{\lambda\vdash d} \frac{(-1)^{\ell(\lambda)}}{z_\lambda} P^+_{\lambda^m}.
\end{align}
\end{theorem}

\begin{theorem}
\label{thm:PPP}
For positive integers $d$ and $m$, we have
\begin{align}
\label{thm:PPtoP}
P^+_{d^m}
&=
\begin{cases}
-P_{d^m}, & \text{if $d$ is odd},\\
-P_{d^m} + 2P_{(d/2)^{2m}}, & \text{if $d$ is even},
\end{cases}
\\
\label{thm:PtoPP}
P_{d^m}
&=
-\sum_{j}
2^j \, P^+_{(d/2^j)^{2^j m}},
\end{align}
where the sum is over all non-negative integers $j$ such that $2^j \mid d$.
\end{theorem}

\begin{proof}
We first prove \eqref{thm:PPtoP}.
If $d$ is odd, then every monomial appearing in $P_{d^m}$ has odd total degree,
so the involution $x_{i,j}\mapsto -x_{i,j}$ sends $P_{d^m}$ to $-P_{d^m}$.
By definition of $P^+$, this gives $P^+_{d^m}=-P_{d^m}$.

Now assume $d$ is even.
Applying $\Omega^+$ to \eqref{eqn:PtoE} yields an expansion of $P^+_{d^m}$ in $E^+$ with the same coefficients as in the expansion of $P_{d^m}$
coming from \eqref{eqn:PdPiecewise}, up to an overall sign.
Comparing these expressions shows that
\begin{align}
P^+_{d^m}
&=
-P_{d^m} + 2P_{(d/2)^{2m}},
\end{align}
which proves \eqref{thm:PPtoP}. To obtain \eqref{thm:PtoPP}, we solve \eqref{thm:PPtoP} for $P_{d^m}$ in the even case,
and substitute repeatedly while the index remains even.
This yields
\begin{align}
P_{d^m}
&=
-\sum_{j}
2^j\,P^+_{(d/2^j)^{2^j m}},
\end{align}
where the sum ranges over all $j\ge 0$ such that $2^j\mid d$.
\end{proof}

\begin{theorem}
\label{thm:PPtoEH}
Let $d$ and $m$ be positive integers. Then
\begin{align}
\label{eqn:PPtoE}
P^+_{d^m}
&=
d\sum_{\substack{r\in\{1,2\}\\ r\mid d}}
\ \sum_{\lambda\vdash d/r}
(-1)^{\ell(\lambda)+r}
\frac{1}{\ell(\lambda)}
\binom{\ell(\lambda)}{m_1(\lambda),m_2(\lambda),\ldots}\,
E_{\lambda^{rm}},\\
\label{eqn:PPtoH}
P^+_{d^m}
&=
d\sum_{\substack{r\in\{1,2\}\\ r\mid d}}
\ \sum_{\lambda\vdash d/r}
(-1)^{\ell(\lambda)+r-1}
\frac{1}{\ell(\lambda)}
\binom{\ell(\lambda)}{m_1(\lambda),m_2(\lambda),\ldots}\,
H_{\lambda^{rm}}.
\end{align}
\end{theorem}
\begin{proof}
Our result follows immediately by combining
Theorem \ref{thm:PtoH} with Theorem \ref{thm:PPP}.
\end{proof}

\begin{definition}
For any stack partition $\tau$, define
\begin{align}
Z_\tau
=
\prod_{d,m} d^{m_{d,m}(\tau)}\, m_{d,m}(\tau)!.
\end{align}
\end{definition}

\begin{example}
Let $\tau
= (3^4,3^3,3^3,2^2,2^2,2^2,2^1,2^1,2^1,2^1).$
Then the nonzero values of $m_{d,m}(\tau)$ are
\begin{align}
m_{3,4}(\tau)=1,\qquad m_{3,3}(\tau)=2,\qquad
m_{2,2}(\tau)=3,\qquad m_{2,1}(\tau)=4.
\end{align}
Hence
\begin{align}
Z_\tau
&=
\prod_{d,m} d^{m_{d,m}(\tau)}\,m_{d,m}(\tau)!
\\
&=
3^{1}1!\cdot 3^{2}2!\cdot 2^{3}3!\cdot 2^{4}4!
\\
&=
995,328.
\end{align}
\end{example}

\begin{theorem}
\label{thm:EHtoPP}
    For positive integers $d$ and $m$, we have
    \begin{align}
    \label{eqn:E_in_Pplus}
        E_{d^m} &= \sum_{\tau\Vdash d} \frac{1}{Z_{\tau}} P^+_{\tau^m},\\
        \label{eqn:H_in_Pplus}
        H_{d^m} &= \sum_{\tau\Vdash d} \frac{(-1)^{\ell(\tau)}}{Z_{\tau}} P^+_{\tau^m},
    \end{align}
    where each multiplicity in $\tau$ is a power of 2.
\end{theorem}
\begin{proof}
We prove \eqref{eqn:E_in_Pplus},  \eqref{eqn:H_in_Pplus} follows by
applying the involution $\Omega$. From \eqref{eqn:GenH=GenP}, we have
\begin{align}
\sum_{d=0}^\infty E_d^+ t^d
&=
\exp\left(-\sum_{k=1}^{\infty}\frac{P_k}{k}t^k\right).
\end{align}
Substituting the expansion of $P_k$ in terms of $P^+$ from
Theorem~\eqref{thm:PtoPP} gives
\begin{align}
\sum_{d=0}^\infty E_d^+ t^d
&=
\exp\left(
\sum_{k=1}^{\infty}
\sum_{j:\,2^j\mid k}
\frac{2^j}{k}\,P^+_{(k/2^j)^{2^j}}\,t^k
\right).
\end{align}
Writing $k=2^j r$ yields $2^j/k=1/r$ and $t^k=t^{2^j r}$, so
\begin{align}
\label{eqn:E_exp_Pplus_clean}
\exp\left(
\sum_{j=0}^{\infty}\sum_{r=1}^{\infty}
\frac{1}{r}\,P^+_{r^{2^j}}\,t^{2^j r}
\right)
=
\prod_{j=0}^{\infty}
\exp\left(
\sum_{r=1}^{\infty}\frac{1}{r}\,P^+_{r^{2^j}}\,t^{2^j r}
\right).
\end{align}
Fix $j\ge 0$, notice that
\begin{align}
\exp\left(
\sum_{r=1}^{\infty}\frac{1}{r}\,P^+_{r^{2^j}}\,t^{2^j r}
\right)=
\prod_{r=1}^{\infty}
\exp\left(
\frac{1}{r}\,P^+_{r^{2^j}}\,t^{2^j r}
\right) = \prod_{r=1}^\infty \sum_{s=0}^{\infty}
\frac{1}{s!}
\left(
\frac{1}{r}\,P^+_{r^{2^j}}\,t^{2^j r}
\right)^s.
\end{align}
Putting everything together gives us 
\begin{align}
    \sum_{d=0}^\infty E_d^+ t^d = \prod_{j=0}^{\infty}\prod_{r=1}^\infty \sum_{s=0}^{\infty}
\frac{1}{s!}
\left(
\frac{1}{r}\,P^+_{r^{2^j}}\,t^{2^j r}
\right)^s.
\end{align}
Thus, choosing a term in the full product is equivalent to choosing, for each
$r$, a nonnegative integer $s$ indicating how many times the
factor $P^+_{r^{2^j}}\,t^{2^j r}$ is selected. Performing this choice independently for all $j\ge 0$ determines a unique stack
partition $\tau$,
in which the term $r^{2^j}$ appears exactly $s=m_{r,2^j}(\tau)$ times.
In particular, every multiplicity in $\tau$ is a power of $2$. Comparing coefficients such that $|\tau|=d$ gives our result, proving $\eqref{eqn:E_in_Pplus}$.
\end{proof}

\section{Final Remarks}
\label{section 6}

This paper develops explicit transition formulas among several natural bases of $\PS$, including $H,E,P$ and their signed variants, extending many classical identities from symmetric function theory to the polysymmetric setting.
The theory of polysymmetric functions was originally motivated in \cite{g2024polysymmetricfunctionsmotivicmeasures} by connections to algebraic geometry via motivic measures, and a full geometric realization of the polysymmetric framework remains an open direction.
We conclude by recording several basic structural questions concerning Schur type bases, inner products, and tableau models in the polysymmetric setting.
\begin{question}
    What is the Schur function analogue in the polysymmetric function setting?
\end{question}
A bilinear form on $\PS$ was introduced in \cite{g2024polysymmetricfunctionsmotivicmeasures}, defined by
\begin{align}
    \langle M_\lambda, H_\tau \rangle = 
\begin{cases}
1 & \text{if } \tau^T = \lambda, \\
0 & \text{otherwise}.
\end{cases}
\end{align}
where, if $\tau = (d_1^{m_1}, \ldots,m_s^{d_s})$ is a stack partition, its transpose $\tau^T$ is defined by $\tau^T = (m_1^{d_1}, \ldots,m_s^{d_s}).$

This serves as a polysymmetric analogue of the classical $\langle m_\lambda, h_\mu \rangle = \delta_{\lambda\mu}$ inner product in $\Lambda$. 

\begin{question}
How does this bilinear form interact with the various natural bases of $\PS$?
In particular, are there bases that are orthogonal or self-dual with respect to this form?
\end{question}

\begin{question}
Is there a natural notion of skewness and of tableaux for stack partitions, analogous to skew shapes and Young tableaux in the classical symmetric function setting?
\end{question}

One can prove that for any positive integer $n$,
\begin{align}
\sum_{\lambda \vdash n}
\frac{n}{\ell(\lambda)}
\binom{\ell(\lambda)}{m_1(\lambda), m_2(\lambda), \ldots}
= 2^n - 1,
\end{align}
where the sum ranges over all integer partitions $\lambda$ of $n$.

We now consider an analogous sum over all \emph{stack partitions} $\tau$ of $n$
\begin{align}
\label{eqn:stackedtablur}
\mathrm{St}(n)=
    \sum_{\tau } \frac{n}{\ell(\tau)} \binom{\ell(\tau)}{m_{1,1}(\tau), m_{1,2}(\tau), \ldots, m_{2,1}(\tau), m_{2,2}(\tau), \ldots}.
\end{align}
\begin{question}
Can one find a closed-form expression for the sum over stack partitions in equation \eqref{eqn:stackedtablur}? We list the first 20 values of this sequence in Table \ref{table:StackedN}.
\end{question}

The definition of a stack partition can be reformulated to yield a new type of partition.
\begin{definition}
Fix a positive integer $d$. A \emph{product-sum partition of order $d$} of $n$ is a finite multiset
\begin{align}
\tau=\Bigl\{ (a^{(1)}_1,\dots,a^{(1)}_d),\ (a^{(2)}_1,\dots,a^{(2)}_d),\ \dots,\ (a^{(k)}_1,\dots,a^{(k)}_d)\Bigr\},
\end{align}
of ordered $d$-tuples of positive integers such that
\begin{align}
\sum_{i=1}^k \ \prod_{j=1}^d a^{(i)}_j \;=\; n.
\end{align}
We write $\mathrm{PSP}_d(n)$ for the set of all such $\tau$. For $n=0$, we take $\mathrm{PSP}_d(0)$ to consist of a single empty multiset.
\end{definition}

Note that $\mathrm{PSP}_1(n)$ recovers ordinary integer partitions, and $\mathrm{PSP}_2(n)$ recovers stack partitions.

\begin{question}
What can be said about the combinatorial structure of $\mathrm{PSP}_d(n)$ and the enumerator $|\mathrm{PSP}_d(n)|$ as $d$ and $n$ vary? 
\end{question}

\begin{table}[h!]

\centering
\caption{The first twenty values of the sequence $\mathrm{St}(n)$.}
\renewcommand{\arraystretch}{1}
\label{table:StackedN}
\begin{tabular}{|c|c||c|c|}
\hline
$n$ & $\mathrm{St}(n)$ & $n$ & $\mathrm{St}(n)$ \\ \hline
1 & 1 & 11 & 20131 \\
2 & 5 & 12 & 49675 \\
3 & 13 & 13 & 122162 \\
4 & 37 & 14 & 300942 \\
5 & 86 & 15 & 740798 \\
6 & 227 & 16 & 1824205 \\
7 & 540 & 17 & 4491095 \\
8 & 1357 & 18 & 11058338 \\
9 & 3316 & 19 & 27226621 \\
10 & 8200 & 20 & 67037152 \\
\hline
\end{tabular}
\end{table}

\begin{table}[h!]
\normalsize
\centering
\renewcommand{\arraystretch}{1.2}
\caption{Polysymmetric Functions Transition Table}
\begin{tabular}{|c|c|c|c|c|c|c|}
\hline
\diagbox[width=7em, height=1.2cm]{From}{To} 
& $H$ & $H^+$ & $E$ & $E^+$ & $P$ & $P^+$ \\ \hline

$H$ 
& $-$ 
& \scalebox{0.75}{Thm. \ref{thm:BtoB+}} 
& \scalebox{0.75}{Thm. \ref{thm:EdHd}} 
& \scalebox{0.75}{Thm. \ref{thm:EtoHP}} 
& \scalebox{0.75}{Thm. \ref{thm:HandP}} 
& \scalebox{0.75}{Thm. \ref{thm:EHtoPP}} 
\\ \hline

$H^+$ 
& \scalebox{0.75}{Thm. \ref{thm:BPtoB}} 
& $-$ 
& \scalebox{0.75}{Thm. \ref{thm:EPtoH}} 
& \scalebox{0.75}{Thm. \ref{thm:PlusToPlus}} 
& \scalebox{0.75}{Thm. \ref{thm:PlustoP}} 
& \scalebox{0.75}{Thm. \ref{thm:PP}} 
\\ \hline

$E$ 
& \scalebox{0.75}{Thm. \ref{thm:EdHd}} 
& \scalebox{0.75}{Thm. \ref{thm:EtoHP}} 
& $-$ 
& \scalebox{0.75}{Thm. \ref{thm:BtoB+}} 
& \scalebox{0.75}{Thm. \ref{thm:HandP}} 
& \scalebox{0.75}{Thm. \ref{thm:EHtoPP}} 
\\ \hline

$E^+$ 
& \scalebox{0.75}{Thm. \ref{thm:EPtoH}}  
& \scalebox{0.75}{Thm. \ref{thm:PlusToPlus}} 
& \scalebox{0.75}{Thm. \ref{thm:BPtoB}} 
& $-$ 
& \scalebox{0.75}{Thm. \ref{thm:PlustoP}} 
& \scalebox{0.75}{Thm. \ref{thm:PP}} 
\\ \hline

$P$ 
& \scalebox{0.75}{Thm. \ref{thm:PtoH}} 
& \scalebox{0.75}{Thm. \ref{thm:PtoPlus}} 
& \scalebox{0.75}{Thm. \ref{thm:PtoH}} 
& \scalebox{0.75}{Thm. \ref{thm:PtoPlus}} 
& $-$ 
& \scalebox{0.75}{Thm. \ref{thm:PPP}} 
\\ \hline

$P^+$ 
& \scalebox{0.75}{Thm. \ref{thm:PPtoEH}} 
& \scalebox{0.75}{Thm. \ref{thm:PP}} 
& \scalebox{0.75}{Thm. \ref{thm:PPtoEH}} 
& \scalebox{0.75}{Thm. \ref{thm:PP}} 
& \scalebox{0.75}{Thm. \ref{thm:PPP}} 
& $-$ 
\\ \hline

\end{tabular}
\end{table}

\bibliographystyle{plain} 
\begin{center}
    \bibliography{references} 
\end{center}
\end{document}